\documentclass[reqno,12pt]{amsart}
\usepackage{amsmath}
\usepackage{amssymb}
\usepackage{amsfonts}
\usepackage{color}
\usepackage{cite}

\newcommand{\red}[1]{\textcolor{red}{#1}}

\newcommand{\sgn}{{\rm sgn}}

\newcommand{\abs}[1]{\left\lvert #1 \right\rvert}

\allowdisplaybreaks
\newtheorem{theorem}{Theorem}[section]

\newtheorem{lemma}[theorem]{Lemma}

\theoremstyle{definition}

\theoremstyle{remark}
\newtheorem{remark}[theorem]{Remark}
\theoremstyle{example}
\newtheorem{example}[theorem]{Example}
\theoremstyle{problem}

\theoremstyle{fact}
\newtheorem{corollary}[theorem]{Corollary}
\numberwithin{equation}{section}

\begin{document}
\title[interior inverse problem]{ Interior inverse problem for global conservative multipeakon solutions of the Camassa-Holm equation}
\author[T. Liu]{Tao Liu}
\thanks{}
\address{School of Mathematics and Statistics, Shaanxi Normal
University, Xi'an 710062, PR China}
\email{liutaomath@163.com}

\author[K. Lyu]{Kang Lyu}
\address{School of Mathematics and Statistics, Nanjing
	University of Science and Technology, Nanjing 210094, PR China}
\email{lvkang201905@outlook.com}


\begin{abstract}
 We consider the interior inverse problem associated with the global conservative {multipeakon} solution of the Camassa-Holm equation. 
Based on the inverse spectral theory on the half-line and the oscillation property of eigenfunctions, some 
{(non)}uniqueness results of {the} interior inverse problem are obtained. In addition,  we give the trace formula,  which connects the global conservative {multipeakon} solution with the corresponding eigenvalues and normalized eigenfunctions.
\end{abstract}

\maketitle

\section{Introduction}
The Camassa-Holm (CH) equation
\begin{equation}
u_t-u_{xxt}=2u_xu_{xx}-3uu_x+uu_{xxx}  \label{eq1.1}
\end{equation}%
has received a lot of attention during the last two decades (cf \cite{cam,cam1, laf, con, con1, con2, gey, len, len1, len2, len3, len4, mad, nat}).
 The isospectral problem \cite{bea1, bea, lun} of the CH equation is the weighted Sturm-Liouville equation
\begin{equation}
-f''(x)+\frac{1}{4}f(x)=z\omega(x,t)f(x), \ x\in \mathbb{R}.  \label{eq1.2}
\end{equation}%
Here $\omega=u-u_{xx}$ is called the weight function, $z\in \mathbb{C}$ is the complex parameter.
Inverse spectral theory for \eqref{eq1.2} is useful for solving the Cauchy problem of the CH equation.

  The multipeakon solution of the CH equation is the solution  of the form
 \begin{equation}
 u(x,t)=\frac{1}{2}\sum_{i=1}^N\omega_i(t)e^{-|x-x_i(t)|}, \label{eq1.3}
\end{equation}%
which represents the interaction of $N$ peaks at position $x_i, i=1,\cdots, N$. In this case,
 \begin{equation}
 \omega(x,t)=u(x,t)-u_{xx}(x,t)=\sum_{i=1}^N \omega_i(t)\delta_{x_i(t)},   \label{eq1.4}
\end{equation}%
where $\delta_{x_i(t)}$ is the Dirac measure at $x_i(t)$. Beals, Sattinger and Szmigielski ~\cite{bea} studied the multipeakon solution by solving the inverse problem of \eqref{eq1.2}. They observed that if $\omega$ is definite, then the  inverse problem  is  solvable for any $t\in \mathbb{R}$. Otherwise, the inverse problem is not always solvable. This causes the collision between some peaks, which leads to {finite-time blow-up resembling} wave breaking. Bressan and Constantin \cite{bre} introduced the pair $(u,\mu)$ to describe the overall picture of the collision, where $\mu$ is a non-negative Borel measure with the absolutely continuous part determined by
$u$ via
\begin{equation}
\mu_{ac}(B,t)= \int_B |u(x,t)|^2+|u_x(x,t)|^2 dx,\ t\in\mathbb{R},   \notag
\end{equation}
for each Borel set $B\in \mathcal{B}(\mathbb{R})$.
The pair $(u,\mu)$ is called the global conservative multipeakon solution (\cite{bre, eck2, hol}). In \cite{eck2}, Eckhardt and Kostenko introduced a generalized isospectral problem
 \begin{equation}
-f''(x)+\frac{1}{4}f(x)=z\omega(x,t)f(x)+z^2v(x,t),\ x\in \mathbb{R}.  \label{eq1.5}
\end{equation}%
In the indefinite case, the inverse problem  is always solvable  within the class of spectral problem \eqref{eq1.5}. {$v(\cdot, t)$  is the singular part of $\mu(\cdot,  t)$ and get switched on exact the time of collisions.}

Previous literature on spectral analysis and inverse spectral analysis for \eqref{eq1.5} always considered the case $v=0$ (cf \cite{bea, bea1,ben, con2, eck3, fu, chu1,chu2, chu3, chu4}).  For example,  In \cite{chu1}, Chu and Meng gave the maximization of eigenvalue ratios. 
If $v=0$, $\omega$ is definite and smooth enough, this problem  can be transformed  into a Sturm-Liouville problem \cite{bea1, len4} and hence direct and inverse spectral  conclusions may be drawn from this observation. We refer to \cite{boy, piv, wei, wei1, ges1, yang, yang2, you}  for some spectral and inverse spectral results  of Sturm-Liouville operators.  Inverse spectral problems for \eqref{eq1.5} with $v\ne0$ were widely studied by Eckhardt and his collaborators ({cf\cite{eck1, eck2, eck5, eck7} and the literature cited therein}). In  \cite{eck2},
Eckhardt and Kostenko  showed that $\omega$ and $v$ are uniquely determined by the spectral measure in the case that $\omega$ and $v$ are discrete and supported on a finite set.
In \cite{eck5}, Eckhardt and Kostenko  obtained some uniqueness results,  if $\omega$ and $v$ are periodic and discrete. In the case that $\omega$ and $v$ are general Borel measures, Eckhardt and Kostenko \cite{eck1}
showed that $\omega$ and $v$ are uniquely determined by the spectral measure under additional conditions on the growth property of solutions. We mention that the problem \eqref{eq1.5} with $f(x)/4$ vanished has been studied in \cite{kre, kre1}.

We want to mention a result about inverse Jacobi matrix problem. In \cite{ges},  Gesztesy and Simon reconstruct{ed} the Jacobi matrix from a diagonal Green's function, which motivates us to consider the following interior inverse problem: For given $a\in\mathbb{R}$, {reconstruct} $\omega, v$ from $\{\lambda_i, \varphi_i(a)\}_{i=1}^N$ (applying Lemma \ref{lemma3.1}), where $\varphi_i(a)$ is the  normalized eigenfunction defined by \eqref{eq2.6}.

There are two main difficulties in solving this interior inverse problem. Firstly, we study the diagonal Green's function $z\phi_+(z,a)\phi_-(z,a)/W(z)$ (see the definitions of $\phi_\pm(z,a), W(z)$ in Section 2) so {that} we can determine the Weyl-Titchmarsh function from the interior spectral data. The diagonal Green's function in \cite{ges} is a rational Herglotz-Nevanlinna function with 
\begin{align}
	c:=\deg(den)-\deg(num)=1,\nonumber
\end{align}
where $\deg(den)$ and $\deg(num)$ are the degrees of the denominator and numerator, respectively. However, there are three cases in our problem: $c=0,1,-1$. {This will yield more complicated results. Moreover, we also obtain various interior inverse spectral conclusions in terms of the signs of eigenvlalues (see Theorem \ref{aaaa} for example).}
Secondly, to determine the Weyl-Titchmarsh functions ($M_\pm(z,a)$) from the interior spectral data, we need to divide the poles of $M_+(z,a)+M_-(z,a)$ into two parts: the poles of $M_+(z,a)$, the poles of $M_-(z,a)$. However, the change of the signs of $\varphi_i(a)$ induce{s} that it is hard to study the poles, especially when $\varphi_i(a)=0$ for some $i$.
This is different from the previous papers {(cf\cite{eck,eck4,fu, wei,ges1, rio,piv,xu, yang ,yang1, yang2, fill, liu, liu1,liu2})}.

The second main result of this paper is the trace formula,  which connects the global conservative multipeakon {solution} $(u,\mu)$ with the corresponding eigenvalues and normalized eigenfunctions of \eqref{eq1.5}. Let $\{\lambda_i\}_{i=1}^{N}$ be the eigenvalues of \eqref{eq1.5}.  For a fixed $t$, let $\{\varphi_i(x,t)\}_{i=1}^{N}$ be the corresponding normalized eigenfunctions. Then we have (see Theorem \ref{theorem 4.1})
\begin{equation}
u(x,t)=\frac{1}{2}\sum_{i=1}^{N}\frac{|\varphi_i(x,t)|^2}{\lambda_i}.   \label{eq1.6}
\end{equation}
It is well-known that $u$ consists of $N$ widely separated peakons in the long time limit. By \eqref{eq1.6} we know that there is one $|\varphi_i|^2/2\lambda_i$ associated with each peakon.  This type of trace formula first appeared in \cite[Theorem 3.4]{gar} for the KdV equation.

 The structure of this paper is as follows. Necessary facts on the generalized spectral problem are collected in Section 2. In Section 3, we are concerned with the inverse interior problem under the condition that
 $\varphi_i(a)\neq0$ for any $i=1,\cdots,N$.  Section 4 considers the inverse interior problem with $\varphi_i(a)=0$ for some $i$. {The complete characterization of the interior spectral data corresponding to discrete measures $\omega$ and $ v$ is given.}
  In Section 5, we give the trace formula \eqref{eq1.6}. By virtue of \eqref{eq1.6}, we give the supremum of $|u(x,t)|$, which depends only on the height(depth) of the highest(deepest) peakon. 
In Section 6, we provide a proof for oscillation of eigenfunctions, which is needed for proving some interior inverse spectral results.

\section{The Generalized Spectral Problem}
In this section, we give some preliminary assertions about the generalized spectral problem. {We will closely follow the notation employed in \cite{eck2}.}

Let $n \in \mathbb{N}$ be fixed.
Suppose that $x_1<x_2<\cdots<x_n$. Define the measures $\omega$ and $v$ by
\begin{equation}
	\omega=\sum_{i=1}^n \omega_i\delta_{x_i}, \ v=\sum_{i=1}^n v_i\delta_{x_i},  \label{eq2.111}
\end{equation}
where $\omega_i\in\mathbb{R}$, $v_i\geq0$,  $\delta_{x_i}$ is the Dirac measure at $x_i,\ i=1,\cdots,n$. We always assume that
\begin{equation}
 |\omega_i|+v_i>0,\  i=1,\cdots, n.  \label{omegai}
\end{equation}

Consider the {generalized spectral problem}
\begin{equation}
Hf=-f''(x)+\frac{1}{4}f(x)=z\omega(x)f(x)+z^2v(x)f(x), \ x\in \mathbb{R},  \label{eq2.1}
\end{equation}
where $\omega$ and  $v$ are defined by \eqref{eq2.111}.
 Note that equation \eqref{eq2.1} has to be interpreted in the sense of distributions, namely, a function $f$ satisfies \eqref{eq2.1} if and only if
\begin{equation}
-f''(x)+\frac{1}{4}f(x)=0, x\in \mathbb{R}\backslash \{x_1,\cdots,x_n\},  \label{eq2.2}
\end{equation}
and for any $i=1,\cdots,n,$
\begin{align}
\label{lianxu}&f(x_i-)=f(x_i+),    \\
\label{tiaoyue} f'(x_i-)=f'(x_i&+)+z\omega_if(x_i+)+z^2v_if(x_i+).
\end{align}
{We} mention that $f$ is differentiable on $\mathbb{R}\setminus\{x_1,\cdots, x_n\}$. We define
\begin{align}
	f'(x_i):=f'(x_i-),\ i=1,\cdots,n,\label{zuolianxu}
\end{align}
so that $f'$ is left-continuous on $\mathbb{R}$.

 The values of the parameter $z$ for which  \eqref{eq2.1} has  nontrivial bounded solutions are called eigenvalues of $H$, and the corresponding nontrivial bounded solutions are called eigenfunctions. All the eigenvalues are real (see for example Theorem \ref{theorem5.9} or \cite{eck2}).
 The set of eigenvalues is called the spectrum of $H$ (denoted by $\sigma(H)$).

  For any $z\in\mathbb{C}$, there exist solutions $\phi_{\pm}(z,\cdot)$ of (\ref{eq2.1}) satisfying
\begin{equation}
 \phi_{\pm}(z,x)=e^{\mp\frac{x}{2}}   \label{eq2.222}
\end{equation}
for all $x$ near $\pm \infty$. The Wronskian of $\phi_{+}(z,x), \phi_{-}(z,x)$, which is independent of $x$, is given by
\begin{equation}\label{Wzdefine}
W(z)=\phi_{+}(z,x)\phi_-'(z,x)-\phi_{+}'(z,x)\phi_-(z,x), z\in\mathbb{C}.
\end{equation}
Note that by \eqref{eq2.1}, \eqref{lianxu} and \eqref{tiaoyue}, $W(z)$ is a polynomial with degree $\leq 2n$.
For $\lambda\in \mathbb{R}$, $W(\lambda)=0$ if and only if $\phi_{+}(\lambda,x)$ and $\phi_{-}(\lambda,x)$ are linear dependent. Namely, there exists a nonzero constant $c_{\lambda}$ such that
\begin{equation}\label{xishu}
\phi_{-}(\lambda,x)=c_{\lambda}\phi_{+}(\lambda,x).
\end{equation}
By \eqref{eq2.222}, we know that the zeros of $W(z)$ coincide with $\sigma(H)$. Therefore, $\sigma(H)$ consists of finitely many eigenvalues. Notice that $0\notin \sigma(H)$ because of
\begin{align}
{W(0)=1}. \label{w(0)}
\end{align}

For any eigenvalue $\lambda$, denote by
\begin{equation}
\gamma_{\lambda}^2:= \int_{\mathbb{R}}|\phi_{+}(\lambda,x)|^2 d  \omega(x)+ 2\lambda\int_{\mathbb{R}}|\phi_{+}(\lambda,x)|^2 d  v(x). \label{eq2.4}
\end{equation}
Applying \eqref{eq2.1}, by some calculations one can obtain
\begin{equation} \nonumber
\lambda\gamma_{\lambda}^2= \frac{1}{4}\int_{\mathbb{R}}|\phi_{+}(\lambda,x)|^2 dx+ \int_{\mathbb{R}}|\phi_{+}'(\lambda,x)|^2 dx+\lambda^2 \int_{\mathbb{R}}|\phi_{+}(\lambda,x)|^2 d  v(x)>0.
\end{equation}
We call $\lambda\gamma_{\lambda}^2$ the (modified) norming constant associated with the eigenvalue $\lambda$.
\begin{lemma}  \label{lemma2.1}
\cite[Proposition 3.1]{eck2} For any $\lambda\in\sigma(H)$, $\lambda$ is real and
\begin{equation}
\dot{W}(\lambda)=-c_{\lambda}\gamma_{\lambda}^2\neq0,   \label{eq2.5}
\end{equation}
where $c_\lambda$ is defined by \eqref{xishu}.
\end{lemma}

 We need to introduce more concepts for studying the interior spectral problem discussed in the following.
 For a fixed $a\in \mathbb{R}$, consider $H_{\pm}(\omega,v,a )$ defined by  restricting \eqref{eq2.1} to the intervals $I_{\pm}(a)$ with Dirichlet boundary condition at $a$, where $I_+(a)=[a,+\infty)$ and $I_-(a)=(-\infty, a)$.
The Weyl-Titchmarsh functions of $H_{+}(\omega,v,a )$ and $H_{-}(\omega,v ,a)$ are defined by
\begin{equation}\label{eq212}
M_+(z, a)=\frac{\phi_+'(z,a)}{{z}\phi_+(z,a)}
 \end{equation}%
 and
  \begin{equation}\label{M-za}
  M_-(z, a)=-\frac{\phi_-'(z,a)}{{z}\phi_-(z,a)},
   \end{equation}%
  respectively. Clearly, ${M_+(z, a)}$, $ {M_-(z, a)}$ are Herglotz-Nevanlinna functions (see {\cite[Lemma 4.9]{eck1}}). Note that  \begin{align}
  \phi_{\pm}(0,a)=e^{\mp\frac{a}{2}}. \label{phi_+(0,a)}
  \end{align}
  Then by \eqref{eq212} and \eqref{M-za}, we have that
  \begin{align}\label{M+0aM-0a}
  {\rm{Res}}_{z=0}M_+(z,a)={\rm{Res}}_{z=0}M_-(z,a)=-\frac{1}{2}.
  \end{align}
Denote by
\begin{equation}
 l_i:=2\left(\tanh\left(\frac{x_{i+1}-a}{2}\right)-\tanh\left(\frac{x_i-a}{2}\right)\right),i=0,\cdots, n,  \label{eq2.18}
  \end{equation}
and
\begin{equation}
m_i(z):=(zv_i+\omega_i)\cosh^2\left(\frac{x_i-a}{2}\right),z\in \mathbb{C}, i=1,\cdots, n,         \label{eq2.19}
  \end{equation}
where  $x_0:=-\infty$ and $x_{n+1}:=+\infty$.

The proofs of the following two lemmas are similar to that of \cite[Lemma 3.4]{eck5}. For readers' convenience, we include a proof for Lemma \ref{lemma2.4}.

\begin{lemma}
		Suppose that $a\in (x_{n_0},x_{n_0+1}]$ for some $n_0$, then the function $M_+(z,a)$ has the following finite continued fraction expansion
	\begin{equation}
	M_+(z,a)=\cfrac{1}{-l(n_0)z+\cfrac{1}{m_{n_0+1}(z)+\cfrac{1}{-l_{n_0+1}z+\cfrac{1}{\ddots+\cfrac{1}{-l_{n-1}z+\cfrac{1}{m_n(z)-\cfrac{1}{l_nz}}}}}}},   \label{eq2.21}
	\end{equation}
	where
	\begin{equation}
	l(n_0)=2\tanh\left(\frac{x_{n_0+1}-a}{2}\right).   \label{eq2.22}
	\end{equation}
\end{lemma}

\begin{remark}
	By \eqref{eq2.21}, we have
	\begin{equation}
	M_+(z,x_{n_0+1})=m_{n_0+1}(z)+\cfrac{1}{-l_{n_0+1}z+\cfrac{1}{m_{n_0+2}(z)+\cfrac{1}{\ddots+\cfrac{1}{-l_{n-1}z+\cfrac{1}{m_n(z)-\cfrac{1}{l_nz}}}}}}.   \label{eq2.23}
	\end{equation}
\end{remark}

\begin{lemma} \label{lemma2.4}
 Suppose that $a\in (x_{n_0},x_{n_0+1}]$ for some $n_0$, then the function $M_-(z,a)$ has the following finite continued fraction expansion
	\begin{equation}
	M_-(z,a)=\cfrac{1}{-\tilde{l}(n_0)z+\cfrac{1}{m_{n_0}(z)+\cfrac{1}{-l_{n_0-1}z+\cfrac{1}{\ddots+\cfrac{1}{-l_{1}z+\cfrac{1}{m_1(z)-\cfrac{1}{l_0z}}}}}}},   \label{eq2.20}
	\end{equation}
	where
	\begin{equation}
	\tilde{l}(n_0)=-2\tanh\Big(\frac{{x_{n_0}}-a}{2}\Big). \label{eq2.200}
	\end{equation}
\end{lemma}

\begin{proof}
	If $a\in (-\infty, x_1]$, it is obvious that  \eqref{eq2.20} holds. We next show that \eqref{eq2.20} holds in the case that $a\in (x_{n_0},x_{n_0+1}]$ with $n_0\ge1$.
	Define
	\begin{equation}
	\Lambda_-(z,x)=-\frac{\phi_-'(z,x)}{\phi_-(z,x)}\cosh^2\left(\frac{x-a}{2}\right)+\frac{1}{2} \sinh\left(\frac{x-a}{2}\right)
	\cosh\left(\frac{x-a}{2}\right).  \label{eq2.17}
	\end{equation}
From  {\eqref{zuolianxu}, \eqref{eq2.222} } and \eqref{eq2.17}, one can obtain
\begin{equation}  \nonumber
\Lambda_-(z,x_1)=-\frac{e^{-(x_1-a)}+1}{4}=-\frac{1}{l_0}.
\end{equation}
Applying \eqref{lianxu}, \eqref{tiaoyue}, \eqref{zuolianxu} and \eqref{eq2.19}, for $i=1,\cdots,n$, one has
\begin{equation}
\Lambda_-(z,x_i+)-\Lambda_-(z,x_i)=(z\omega_i+z^2v_i)\cosh^2\left(\frac{x_i-a}{2}\right)=zm_i(z).   \label{m_i}
\end{equation}
{Differentiating} \eqref{eq2.17} with respect to $x$, after some calculations, we have that for any $x\in \mathbb{R} \setminus\{x_1,\cdots, x_n\}$,
\begin{equation}
\frac{\Lambda_-'(z,x)}{\Lambda_-(z,x)^2}=\frac{1}{\cosh^2\left(\frac{x-a}{2}\right)}.  \label{Lambda}
\end{equation}
Integrating \eqref{Lambda} from $x_i+$ to $x_{i+1}$, and from $x_{n_0}+$ to $a$, we conclude that
\begin{align}
\frac{1}{\Lambda_-(z,x_{i+1})}-\frac{1}{\Lambda_-(z,x_i+)}&=-2\tanh\left(\frac{x-a}{2}\right)\Big|_{x_i}^{x_{i+1}}=-l_i,  \label{eq2.25} \\
\frac{1}{\Lambda_-(z,a)}-\frac{1}{\Lambda_-(z,x_{n_0}+)}&=-2\tanh\left(\frac{x-a}{2}\right)\Big|_{x_{n_0}}^{a}=-\tilde{l}(n_0).    \label{eq2.26}
\end{align}
Thus from \eqref{m_i}, \eqref{eq2.25} and \eqref{eq2.26}, we have
\begin{equation}
\frac{1}{\Lambda_-(z,x_{i+1})}=-l_i+\frac{1}{zm_i(z)+\Lambda_-(z, x_{i})},    \nonumber
\end{equation}
and
\begin{equation}
\frac{1}{\Lambda_-(z,a)}=-\tilde{l}(n_0)+\frac{1}{zm_{n_0}(z)+\Lambda_-(z,x_{n_0})}.  \nonumber
\end{equation}
Note that
\begin{equation} \nonumber
M_-(z,a)=\frac{\Lambda_-(z,a)}{z}.
\end{equation}
Therefore, one can obtain \eqref{eq2.21}.
\end{proof}
\begin{remark}
	By Lemma \ref{lemma2.4}, we have that
 \begin{align}\label{limM-zaoverz}
\lim_{\abs{z}\to+\infty}M_-(z,a)=0.
\end{align}		
\end{remark}
The proofs of the following two lemmas are similar to that of \cite[Theorem 5.1]{eck5}. For reader's convenience, we give a proof of Lemma \ref{lemma2.7}.
\begin{lemma}  \label{lemma2.6}
	Let $M_+(z)$ be a rational Herglotz-Nevanlinna function with
	\begin{align}\label{M+0a1}
	{\rm{Res}}_{z=0}M_+(z)=-\frac{1}{2}.
	\end{align}
	Then 	for any $a\in \mathbb{R}$, there exists unique $\omega|_{[a,+\infty)}$, $v|_{[a,+\infty)}$ of the form $\eqref{eq2.111}$, so that  the  Weyl-Titchmarsh function of $H_+(\omega, v, a)$ coincides with $M_+(z)$.
\end{lemma}

\begin{lemma} \label{lemma2.7}
	Let $M_-(z)$ be a rational Herglotz-Nevanlinna function with
	\begin{align}\label{limM-zaover}
	\lim_{\abs{z}\to+\infty}M_-(z)=0
	\end{align}
	and
\begin{align}\label{M+0a2}
{\rm{Res}}_{z=0}M_-(z)=-\frac{1}{2}.
\end{align}
Then for any $a\in \mathbb{R}$, there exists  unique $\omega|_{(-\infty,a)}$, $v|_{(-\infty,a)}$ of the form $\eqref{eq2.111}$, so that  the  Weyl-Titchmarsh function of $H_-(\omega, v, a)$ coincides with $M_-(z)$.
\end{lemma}

\begin{proof}
{By \eqref{limM-zaover} and \cite[Corollarys 4.3 and 4.4]{eck6},} $M_-(z)$ has a finite continued fraction expansion of the form \eqref{eq2.20} for some positive $\tilde{l}(n_0), l_{n_0-1},\cdots, l_0,$  and some real, nonzero polynomials $m_i(z)$ of degree at most one with $\dot{m}_i(0)\ge0 $ for all $i=1,\cdots,n_0$. 	Note that
\begin{equation} \label{11111}
\tilde{l}(n_0)+\sum_{j=0}^{n_0-1} l_j=-\lim_{z\rightarrow 0}  \frac{1}{zm(z)}=2.
\end{equation}
Then for any $a\in \mathbb{R}$, we could define $x_{n_0},\cdots, x_{1},$ by
	\begin{align}
	-2\tanh\left(\frac{{x_{n_0}}-a}{2}\right)&=\tilde{l}(n_0), \label{eq2.2000}   \\
	-2\tanh\left(\frac{{x_{i}}-a}{2}\right)&=\tilde{l}(n_0)+\sum_{j=i}^{n_0-1} l_j, i={n_0-1},\cdots,1.\label{eq2.20000}
	\end{align}
By \eqref{11111}, \eqref{eq2.2000} and  \eqref{eq2.20000}, we can conclude that  \eqref{eq2.18} holds for $i=0,\cdots,n_0-1$. For $i=1,\cdots,n_0$, define $\omega_i, v_i$ such that \eqref{eq2.19} holds. Define
\begin{equation}
	\omega=\sum_{i=1}^{n_{0}} \omega_i\delta_{x_i}, \ v=\sum_{i=1}^{n_{0}} v_i\delta_{x_i}.
\end{equation}
By the above construction, we know that the Weyl-Titchmarsh function of $H_-(\omega,v, a)$ coincides with $M_-(z)$.
Since the fraction expansion of the form \eqref{eq2.20} for $M_-(z)$ is unique, then $\omega|_{(-\infty,a)}$ and $v|_{(-\infty,a)}$ are uniquely determined.
\end{proof}

\section{Interior Inverse Problem I}
In this section we consider the interior inverse problem with the case that $\varphi_i(a)\neq0$ (see \eqref{eq2.6} for the definition of $\varphi_i(a)$) for any $i=1,\cdots,N$.
Since $\sigma(H)$ consists of finitely many eigenvalues, we suppose that
\begin{align}
	\sigma(H)=\{\lambda_i\}_{i=1}^N,\nonumber
\end{align}
where $N\leq 2n$, $\lambda_1<\lambda_2<\cdots<\lambda_N$.
We mention that $0\notin \sigma(H).$ 
Define
\begin{equation}
\varphi_i(x)=\frac{\phi_{+}(\lambda_i,x)}{\sqrt{\lambda_i\gamma_{\lambda_i}^2}}      \label{eq2.6}
\end{equation}
as the normalized eigenfunction associated with the eigenvalue $\lambda_i$.
 We consider the following interior inverse problem: For given $a \in \mathbb{R}$, reconstruct $\omega$ and $v$ from
$\{\lambda_i, \varphi_i(a) \}_{i=1}^N$.
The set $\{\lambda_i, \varphi_i(a) \}_{i=1}^N$ is called the interior spectral data of $H=H(\omega,v)$.  We mention that  $\omega$ and $v$ are uniquely determined by the spectral data  $\{\lambda_i, \lambda_i\gamma_{\lambda_i}^2\}_{i=1}^N$ (\cite[Theorem 4.3]{eck2}).

  Define
\begin{equation}
G(z,a)=\sum_{i=1}^N \frac{\lambda_i^2|\varphi_i(a)|^2}{\lambda_i-z}.  \label{G}
 \end{equation}%
 Obviously, $G(z,a)$ is a Herglotz-Nevanlinna function. Since $0\notin \sigma(H)$, {$G(z,a)$ has the following Taylor series at $z=0$,}
\begin{equation}\label{TaylorofGza}
G(z,a)=\sum_{i=1}^N \lambda_i |\varphi_i(a)|^2 + z \sum_{i=1}^N |\varphi_i(a)|^2 + z^2 \sum_{i=1}^N \frac{|\varphi_i(a)|^2}{\lambda_i}+O(z^3).
\end{equation}

Let us introduce some basic lemmas. Denote by
\begin{align}
	\alpha:=1-\sum_{i= 1}^N |\varphi_i(a)|^2,\
	\beta:=-\sum_{i=1}^N \lambda_i|\varphi_i(a)|^2.\nonumber
\end{align}

 \begin{lemma}  \label{lemma3.1}
 We have
 \begin{equation}
-\frac{1}{M_+(z,a)+M_-(z,a)}= \alpha z+ \beta + G(z,a) \label{eq3.2}
\end{equation}
with $\alpha\geq0.$
\end{lemma}
\begin{proof}
 By the definitions of $M_+(z,a)$ and $M_-(z,a)$, we have
\begin{align}
-\frac{1}{M_+(z,a)+M_-(z,a)}=\frac{z\phi_+(z,a)\phi_-(z,a)}{W(z)}.    \label{eq3.4}
\end{align}
Combining Lemma 2.1 with \eqref{xishu} and \eqref{eq2.6}, we can obtain
\begin{align}\label{Resz=lambdai}
{\rm{Res}}_{z=\lambda_i} \frac{z\phi_+(z,a)\phi_-(z,a)}{W(z)}
=\frac{\lambda_i\phi_+(\lambda_i,a)\phi_-(\lambda_i,a)}{\dot{W}(\lambda_i)}=-\lambda_i^2|\varphi_i(a)|^2.
\end{align}
{Since $M_+(z,a), M_-(z,a)$ are Herglotz-Nevanlinna functions,
$$-\frac{1}{M_+(z,a)+M_-(z,a)}$$
is also a Herglotz-Nevanlinna function.}
Combining with \eqref{Resz=lambdai}, there exist $t\geq0$, $s\in\mathbb{R}$ such that
 \begin{equation}
 -\frac{1}{M_+(z,a)+M_-(z,a)}= t z+ s + G(z,a).   \label{A}
\end{equation}
Let $z\to0$ in \eqref{A}, we obtain $s=\beta$.
From \eqref{w(0)} and \eqref{phi_+(0,a)}, letting $z\rightarrow0$, one can obtain
\begin{equation}\label{zpp}
\frac{z\phi_+(z,a)\phi_-(z,a)}{W(z)}=z+O(z^2).
\end{equation}
By \eqref{TaylorofGza}, \eqref{eq3.4}, \eqref{A} and  \eqref{zpp}, we conclude that
\begin{equation}
t=\alpha\geq 0. \notag
 \end{equation}
\end{proof}

\begin{lemma} \label{lemma3.2}
	We have
	\begin{equation}
		\alpha=0   \nonumber
	\end{equation}
	if and only if $a=x_{i_0}$ for some $i_0\in\{1,\cdots,n\}$. Furthermore, we have
	\begin{align}
		\beta=0
	\end{align}
if and only if $v_{i_0}>0.$
	
\end{lemma}
\begin{proof}
From \eqref{eq2.19},  \eqref{eq2.21}  and  \eqref{eq2.23}, one has that
$a=x_{i_0}$ for some $i_0\in\{1,\cdots,n\}$ if and only if
there exists $C\in \mathbb{R}$ with
\begin{align} \label{aaa}
	\lim_{\abs{z}\to+\infty}\frac{1}{M_+(z,a)}=C.
\end{align}
 By  \eqref{limM-zaoverz},  we have that \eqref{aaa} holds if and only if
\begin{equation}
\lim_{|z|\rightarrow+\infty}-\frac{1}{M_+(z,a)+M_-(z,a)}=-C,   \notag
\end{equation}
{which is equivalent to $\alpha=0$ by $\eqref{eq3.2}$}.  
Furthermore, by \eqref{eq2.19} and \eqref{eq2.23}, $v_{i_0}>0$ if and only if $C=0$,
{which is equivalent to $\alpha=\beta=0$ by $\eqref{eq3.2}$}.  
\end{proof}

Suppose that $A_1(z),A_2(z),B_1(z),B_2(z)$ are real polynomials such that
\begin{align}
	 F_1(z)=\frac{ A_{1}(z)}{B_{1}(z)},\  F_2(z)=\frac{ A_{2}(z)}{B_{2}(z)}\label{F1zF2z}
\end{align}
are Herglotz-Nevanlinna functions. Assume that $A_i(z)$  and $B_i(z)$ have no common zeros, $i=1,2$. Define
\begin{equation}
	F(z):=F_1(z)+ F_2(z).   \label{definitionFz}
\end{equation}
Clearly, for any adjacent zeros of $F(z)$ (denoted by $\nu_j,\nu_{j+1}$), there exists exactly one pole $\mu$ of $F(z)$ with $ \nu_j<\mu<\nu_{j+1}$.

We need the following lemma.
\begin{lemma} \label{lemma3.3}
Assume that $\mu$ is not a common pole of $F_1(z)$ and $F_2(z)$. If
\begin{align}
	B_{2}(\nu_{j})B_{2}(\nu_{j+1})>0, \label{tonghao1}
\end{align}
then $\mu$ is a pole of $F_{1}(z)$. Otherwise, $\mu$ is a pole of $F_{2}(z)$.
\end{lemma}
\begin{proof} Note that for any $\lambda\in [\nu_j, \nu_{j+1}] \setminus \{\mu\}$, \begin{align}
	\ B_{i}(\lambda)\neq 0, \quad  i=1,2. \label{B1}
	\end{align}
Otherwise, for example, assume $B_1(\lambda)=0$. Since polynomials $A_{1}(z)$  and $B_{1}(z)$ have no common zeros, one has $A_1(\lambda)\neq 0$. Because $F_1(z)$ is a Herglotz-Nevanlinna function, one has
\begin{align}
	{\rm{Res}}_{z=\lambda}F_1(z)<0.\nonumber
\end{align}
Then by the fact that $F(\lambda)\in \mathbb{R}$ we obtain that $B_2(\lambda)=0$ and
\begin{align}
	{\rm{Res}}_{z=\lambda}F_2(z)>0,\nonumber
\end{align}
this is {in contradiction with} that $F_2(z)$ is a Herglotz-Nevanlinna function.

From \eqref{definitionFz}, we have that
\begin{equation}
F(z)=\frac{f(z)}{B_{1}(z)B_{2}(z)},        \label{Gz}
\end{equation}
where
\begin{equation}
f(z)= A_1(z)B_2(z)+A_2(z)B_1(z).\nonumber
\end{equation}
 Since $F(z)$ is a Herglotz-Nevanlinna function, its zeros are simple. By \eqref{B1} and \eqref{Gz},  we know that $\nu_j$, $\nu_{j+1}$ are simple zeros of $f(z)$. For  any $\lambda\in (\nu_j, \nu_{j+1}) \setminus \{\mu\}$,
combining \eqref{B1} with the fact that $F(\lambda)\ne 0$, we can conclude that $f(\lambda)\ne 0$. By the assumption, $\mu$ is not a common pole of $F_1(z)$ and $F_2(z)$. Then by \eqref{F1zF2z}, $\mu$ is a simple zero of $B_{1}(z)B_{2}(z)$. Since $\mu$ is a pole of $F(z)$, by \eqref{Gz}, one obtains that $f(\mu)\neq0$. Hence
 one can conclude that $\nu_j$ and $\nu_{j+1}$ are adjacent simple zeros of $f(z)$. Then one has that
 \begin{equation}\dot{f}(\nu_{j})\dot{f}(\nu_{j+1})<0.  \label{dotg}
 \end{equation}

 Combining  \eqref{Gz} with the fact that $F(z)$ is a Herglotz-Nevanlinna function, we obtain
\begin{align}
\dot{F}(\nu_j)&=\frac{\dot{f}(\nu_j)
	}{B_{1}(\nu_j)B_{2}(\nu_j)}>0, \label{lambdaj}\\ \dot{F}(\nu_{j+1})&=\frac{\dot{f}(\nu_{j+1})}{B_{1}(\nu_{j+1})B_{2}(\nu_{j+1})}>0.  \label{lambdaj+1}
\end{align}
 Then from  \eqref{dotg}, \eqref{lambdaj} and \eqref{lambdaj+1},  one has
\begin{align}
B_{1}(\nu_j)B_{2}(\nu_j)B_{1}(\nu_{j+1})B_{2}(\nu_{j+1})<0.\nonumber
\end{align}
Applying \eqref{tonghao1}, we have that $B_{1}(\nu_j)B_{1}(\nu_{j+1})<0$. Therefore, $B_{1}(z)$ has a zero in interval $(\nu_{j}, \nu_{j+1})$, which implies that $\mu$ is a pole of $F_1(z)$.

{Otherwise, if $B_{2}(\nu_{j})B_{2}(\nu_{j+1})\leq0$, by \eqref{B1}, one has}
\begin{align}
	B_{2}(\nu_{j})B_{2}(\nu_{j+1})<0.\nonumber
\end{align}
Then $B_{2}(z)$ has a zero in the interval $(\nu_{j}, \nu_{j+1})$. This implies that $\mu$ is a pole of $F_{2}(z)$.
\end{proof}

Now we are going to introduce some inverse interior spectral results under the condition that
$\varphi_i(a)\neq0$ for any $i=1,\cdots,N$.
\begin{theorem}\label{theorem3.3}
For given $a$, suppose that $\{\lambda_i, \varphi_i(a) \}_{i=1}^N$ is the interior spectral data for some $\omega,v$. If
\begin{equation}
\alpha=\beta=0 \nonumber
\end{equation}
and
\begin{equation}\label{phiinotequal0}
\varphi_i(a)\neq0,\ i=1,\cdots, N,
\end{equation}
then $\{\lambda_i, \varphi_i(a) \}_{i=1}^N$
uniquely determines $\omega,v$.
\end{theorem}
\begin{proof}
Since $\alpha=\beta=0,$ by \eqref{G} and Lemma \ref{lemma3.1}, we have
\begin{equation}\label{sumi1N}
-\frac{1}{M_+(z,a)+M_-(z,a)}=\sum_{i=1}^N \frac{\lambda_i^2|\varphi_i(a)|^2}{\lambda_i-z}.
\end{equation}
Namely, $-\frac{1}{M_+(z,a)+M_-(z,a)}$ is completely determined by the interior spectral data $\{\lambda_i, \varphi_i(a) \}_{i=1}^N$.
Since $M_+(z,a)+M_-(z,a)$ is a Herglotz-Nevanlinna function, by {\eqref{phiinotequal0}} and \eqref{sumi1N}, there exist $\gamma>0, \zeta, \beta_i,\mu_i, i=1,\cdots,N-1$, (~determined by $\{\lambda_i, \varphi_i(a) \}_{i=1}^N$ ) so that
\begin{equation}\label{gammazeta}
M_+(z,a)+M_-(z,a)=\gamma z+\zeta+\sum_{i=1}^{N-1}\frac{\beta_i}{\mu_i-z}, \ {\beta_i>0},
\end{equation}
where $\{\mu_i\}_{i=1}^{N-1}$ are the poles of ${M_+(z,a)}+{M_-(z,a)}$. {We emphasize that ${M_+(z,a)}+{M_-(z,a)}$ has one more zero than poles.}

By \eqref{M+0aM-0a}, we have $\mu_{j_0}=0$ for some {$j_0\in \{1,\cdots, N-1\}$}.
Clearly, $\{\mu_i\}_{i=1}^{N-1}$ and $\{\lambda_i\}_{i=1}^{N}$ are interlaced:
\begin{equation}
\lambda_1<\mu_1<\lambda_2<\cdots<\lambda_{j_0}<\mu_{j_0}=0<\lambda_{j_0+1}<\cdots<\lambda_N. \label{lambda1}
\end{equation}

{Now, we are going to determine the poles of ${M_+(z,a)}$ and $M_-(z,a)$, respectively}. {Then by \eqref{gammazeta} we can completely determine $M_+(z,a)$ and $M_-(z,a).$} Note that the only common pole of $M_+(z,a)$ and $M_-(z,a)$ is $z=0$. Indeed,
if $\mu_{j}\neq0$ is a pole of $M_+(z,a)$ and $M_-(z,a)$. Namely, $$\phi_{+}(\mu_{j},a)=\phi_-(\mu_{j},a)=0.$$
Then by \eqref{Wzdefine} we have that
\begin{equation}
	W(\mu_{j})=\phi_{+}(\mu_{j},a)\phi_-'(\mu_{j},a)-\phi_{+}'(\mu_{j},a)\phi_-(\mu_{j},a)=0.\nonumber
\end{equation}
This implies that $\mu_j$ is an eigenvalue, which leads to {a} contradiction.

Let
\begin{align}
A_1(z)&=-\phi_-'(z,a), \ B_1(z)=z\phi_-(z,a), \nonumber\\
A_2(z)&=\phi_+'(z,a), \ B_2(z)=z\phi_+(z,a),  \nonumber
\end{align}
in Lemma \ref{lemma3.3}. Then we have
\begin{equation}
F_1(z)=M_-(z,a), F_2(z)=M_+(z,a).
\end{equation}
For any {$j\neq j_0$} with
\begin{align}
	\varphi_{j}(a) \varphi_{j+1}(a)<0,\nonumber
\end{align}
from \eqref{eq2.6} and \eqref{lambda1}, one has
\begin{align}
B_2(\lambda_j)B_2(\lambda_{j+1})= \lambda_j\lambda_{j+1}\phi_+(\lambda_{j},a)\phi_+(\lambda_{j+1},a)<0.\nonumber
\end{align}
By Lemma \ref{lemma3.3},  we obtain that $\mu_{j}$ is a pole of $M_+(z,a)$. Similarly, for any {$j\neq j_0$} with $\varphi_{j}(a) \varphi_{j+1}(a)>0,$
$\mu_{j}$ is a pole of $M_-(z,a)$.
 Hence, we can completely determine, by the signs of $\varphi_i(a), i=1,\cdots,N$,  all poles of $M_+(z,a)$ and  $M_-(z,a) $, respectively. Denote the poles of $M_+(z,a)$ and $M_-(z,a)$ by $\{0\}\cup\{u_{i}\}_{i\in I_1}$ and $\{0\}\cup\{u_{i}\}_{i\in I_2}$, respectively. Firstly, {by \eqref{M+0aM-0a}} and \eqref{limM-zaoverz}, one has
\begin{align}
	M_-(z,a)=-\frac{1}{2z}+\sum_{i\in I_2}\frac{\beta_i}{\mu_i-z}.\nonumber
\end{align}
Using \eqref{M+0aM-0a} and \eqref{gammazeta} we obtain
\begin{align}
M_+(z,a)=\gamma z+ \zeta-\frac{1}{2z}+\sum_{i\in {I_1}} \frac{\beta_i}{\mu_i-z}.\nonumber
\end{align}
By {Lemmas \ref{lemma2.6}} and \ref{lemma2.7}, $M_+(z,a)$ {and} $M_-(z,a)$ uniquely determine $\omega|_{[a,+\infty)}$, $v|_{[a,+\infty)}$ and $\omega|_{(-\infty,a)}, v|_{(-\infty,a)}$, respectively. Hence $\omega$ and $v$ are uniquely determined.

Now we show that the $\omega, v$ determined above has the interior spectral data $\{\lambda_i, \varphi_i(a) \}_{i=1}^N$. Assume that $\{\hat{\lambda}_i, \hat{\varphi}_i(a) \}_{i=1}^N$ is the interior spectral data for the determined $\omega, v$. 
{Firstly, $M_{\pm}(z,a)$ are the Weyl-Titchmarsh functions of $H_{\pm}(\omega,v,a)$, then by Lemma \ref{lemma3.1} we have
\begin{equation}
-\frac{1}{M_+(z,a)+M_-(z,a)}=\sum_{i=1}^N \frac{\hat\lambda_i^2|\hat\varphi_i(a)|^2}{\hat\lambda_i-z}. \nonumber
\end{equation}
Combining with \eqref{sumi1N}, for any $i=1,\cdots,N$,   one has
\begin{equation} \label{f}
\hat{\lambda}_i=\lambda_i, |\hat{\varphi}_i(a)|=|\varphi_i(a)|.  
\end{equation}
Therefore, it suffices to} show that for any $i=1,\cdots, N$, we have
\begin{align}\label{sgnequal}
	{\rm{sgn}}(\hat{\varphi}_i(a))={\rm{sgn}}(\varphi_i(a)).
\end{align} 
{By \eqref{lambda1} and Remark \ref{remark5.11}, one has \eqref{sgnequal} {holds} for $i=j_0, j_0+1$.
For any $j\ne j_0$, we know that $\mu_{j}$ is a pole of $M_{\pm}(z,a)$ if and only if
 	 $$\pm\hat{\varphi}_{j}(a)\hat{\varphi}_{j+1}(a)<0.	$$
Hence, the signs of $\hat{\varphi}_{j}(a)$ for $j\neq j_0,j_0+1,$ are uniquely determined by ${\rm{sgn}}(\hat{\varphi}_{j_0}(a))$ and ${\rm{sgn}}(\hat{\varphi}_{j_0+1}(a))$.  Therefore, we can obtain that \eqref{sgnequal} {holds} for any ${i}=1,\cdots,N$.
Then $\{\lambda_i, \varphi_i( a) \}_{i=1}^N$ is the interior spectral data of the determined $\omega, v$.}
\end{proof}

\begin{theorem} \label{aaaa}
 For given $a$, suppose that $\{\lambda_i, \varphi_i(a) \}_{i=1}^N$ is the interior spectral data for some $\omega,v$. Assume that
\begin{equation}
\alpha=0, \ \beta\neq0      \nonumber
\end{equation}
and
\begin{equation}
\varphi_i(a)\neq0 , i=1,\cdots, N.   \label{varphii}
\end{equation}
${\rm{(i)}}$ If we have $\lambda_1>0$ or $\lambda_N<0$,
 then $\{\lambda_i, \varphi_i(a) \}_{i=1}^N$ uniquely determine{s} $\omega,v$.\\
${\rm{(ii)}}$ If for some {$j_0\in \{1,\cdots, N-1\}$} we have
\begin{equation}
\lambda_{j_0}<0<\lambda_{j_0+1},   \notag
\end{equation}
 then  two $\omega,v$'s																																																																																																																																				    yield the given interior spectral data $\{\lambda_i, \varphi_i(a) \}_{i=1}^N$.
\end{theorem}
\begin{proof}
{Arguing as in Theorem \ref{theorem3.3}}, there exist $\zeta, \beta_i,\mu_i, i=1,\cdots,N$, (~determined by $\{\lambda_i, \varphi_i(a) \}_{i=1}^N$ ) so that
\begin{equation}\label{zeta2}
M_+(z,a)+M_-(z,a)=\zeta+\sum_{i=1}^{N}\frac{\beta_i}{\mu_i-z}, \ {\beta_i>0},
\end{equation}
where $\{\mu_i\}_{i=1}^{N}$ are the poles of ${M_+(z,a)}+{M_-(z,a)}$. {We emphasize that ${M_+(z,a)}+{M_-(z,a)}$ has the same number of zeros and poles.}

Let us show ${\rm{(i)}}$ first. We only consider the case that $\lambda_1>0$, the proof is similar for the case that $\lambda_N<0$. Clearly, $\{\mu_i\}_{i=1}^N$ and $\{\lambda_i\}_{i=1}^N$ are interlaced:
\begin{equation}
\mu_1=0<\lambda_1<\mu_2<\lambda_2<\cdots<\mu_{N-1}<\lambda_{N-1}<\mu_N<\lambda_N.
\end{equation}
 Following the {similar} steps in the proof of Theorem \ref{theorem3.3},  we can completely determine all poles of $M_+(z,a)$ and  $M_-(z,a)$, respectively, by the signs of $\varphi_i(a), i=1,\cdots,N$. Then by \eqref{M+0aM-0a},  \eqref{limM-zaoverz} and \eqref{zeta2}, $M_+(z,a)$ and $M_-(z,a)$ are uniquely determined. Hence, we can obtain  {exactly} one acceptable $\omega, v$.

Now we are going to prove ${\rm(ii)}$. Since $\{\mu_i\}_{i=1}^N$ and $\{\lambda_i\}_{i=1}^N$ are interlaced, we have either
\begin{equation}
\lambda_1<\mu_1<\cdots<\lambda_{j_0}<\mu_{j_0}=0<\lambda_{j_0+1}<\cdots<\lambda_N<\mu_N,   \label{lambda11}
\end{equation}
or
\begin{equation}
\mu_1<\lambda_1<\cdots<\lambda_{j_0-1}<\mu_{j_0}=0<\lambda_{j_0}<\cdots<\mu_{N}<\lambda_N.  \label{mu1}
\end{equation}
We only consider the case  that \eqref{lambda11} holds.  Following the similar steps in the proof of Theorem \ref{theorem3.3}, we can completely determine all poles of $M_+(z,a)$ and  $M_-(z,a)$, respectively, except for $\mu_N$. Since  $\mu_N$ can  be a pole of $M_+(z,a)$ or $M_-(z,a)$, we have two acceptable $\omega,v$'s. Therefore, we can conclude that {two $\omega, v$}'s yield the given interior spectral data $\{\lambda_i, \varphi_i(a) \}_{i=1}^N$. 
\end{proof}

\begin{theorem} \label{theorem3.7}
 For given $a$, suppose that $\{\lambda_i, \varphi_i(a) \}_{i=1}^N$ is the interior spectral data for some $\omega,v$.
Assume that $\alpha\neq0$ and
\begin{equation}
\varphi_i(a)\neq0 , i=1,\cdots, N. \notag
\end{equation}
$\rm{(i)}$ If we have
$\lambda_1>0 $ or $\lambda_N<0$,
  then  two $\omega, v$'s yield the given interior spectral data $\{\lambda_i, \varphi_i(a) \}_{i=1}^N$.\\
$\rm{(ii)}$ If for some $j_0\in \{1,\cdots, N-1\}$ we have
\begin{equation}
\lambda_{j_0}<0<\lambda_{j_0+1},   \notag
\end{equation}
 then  four $\omega, v$'s yield the given interior spectral data $\{\lambda_i, \varphi_i(a) \}_{i=1}^N$.
\end{theorem}

\begin{proof}
Since $M_+(z,a)+M_-(z,a)$ is a Herglotz-Nevanlinna function, by $\alpha\neq0$ and Lemma 3.1, there exist $ {\beta_i>0},\mu_i, i=0,\cdots,N$, (~determined by $\{\lambda_i, \varphi_i(a) \}_{i=1}^N$ ) so that
\begin{equation}
M_+(z,a)+M_-(z,a)=\sum_{i=0}^N \frac{\beta_i}{\mu_i-z},      \label{N+1}
\end{equation}
where $\{\mu_i\}_{i=0}^{N}$ are the poles of ${M_+(z,a)}+{M_-(z,a)}$. {We emphasize that ${M_+(z,a)}+{M_-(z,a)}$ has one more pole than zeros.}

Let us show ${\rm{(i)}}$ first. We only consider the case that $\lambda_1>0$.
Clearly, $\{\mu_i\}_{i=0}^N$ and $\{\lambda_i\}_{i=1}^N$ are interlaced:
\begin{equation}
\mu_0=0<\lambda_1<\mu_1<\lambda_1<\cdots<\lambda_{N-1}<\mu_{N-1}<\lambda_N<\mu_N.  \nonumber
\end{equation}
    Then we can completely determine all poles of $M_+(z,a)$ and  $M_-(z,a)$, respectively, except for $\mu_N$.
    Since  $\mu_N$ can  be a pole of $M_+(z,a)$ or $M_-(z,a)$, we have two acceptable $\omega,v$'s. Therefore, we can conclude that {two $\omega, v$}'s yield the given interior spectral data $\{\lambda_i, \varphi_i(a) \}_{i=1}^N$.

 Now we are going to prove ${\rm(ii)}$. Clearly, $\{\mu_i\}_{i=0}^N$ and $\{\lambda_i\}_{i=1}^N$ are interlaced:
 \begin{equation}
\mu_0<\lambda_1<\cdots<\lambda_{j_0}<\mu_{j_0}=0<\lambda_{j_0+1}<\cdots<\mu_{N-1}<\lambda_N<\mu_N.  \nonumber
\end{equation}
 Hence,  we can completely determine all poles of $M_+(z,a)$ and  $M_-(z,a)$, respectively, except for $\mu_0, \mu_N$.  Since  $\mu_0, \mu_N$ can  be the  poles of $M_+(z,a)$ or $M_-(z,a)$, we have four acceptable $\omega,v$'s.
  Therefore, we can conclude that {four $\omega, v$}'s yield the given interior spectral data $\{\lambda_i, \varphi_i(a) \}_{i=1}^N$.
\end{proof}

\begin{example} Suppose that $\sigma(H)=\{\lambda_1\}$, {thus there exist $c, d\in \mathbb{R}$}, so that  $W(z)=cz+d$.
 Then by $\eqref{lianxu}, \eqref{tiaoyue}$ and $\eqref{Wzdefine}$, one obtains that  $$\omega=\omega_1\delta_{x_1},\  v=0$$ for some $\omega_1,x_1$. After some calculations, we have
\begin{eqnarray} \label{zomega1}
\phi_+(z,x)=
\begin{cases}
z\omega_1e^{\frac{x}{2}-x_1}+(1-z\omega_1)e^{-\frac{x}{2}},      & x<x_1, \\
 e^{-\frac{x}{2}},  & x\geq x_1. \\
\end{cases}
\end{eqnarray}
Then one has $1-\lambda_1 \omega_1 =0$. Namely,
\begin{equation}
\omega_1=\frac{1}{\lambda_1}.   \label{omega1}
\end{equation} By \eqref{eq2.4}, \eqref{eq2.6}, \eqref{zomega1} and \eqref{omega1}, we can obtain
\begin{eqnarray} \label{varphi1}
\varphi_1(x)=
\begin{cases}
e^{\frac{x-x_1}{2}},      & x<x_1, \\
 e^{\frac{x_1-x}{2}},  & x\geq x_1. \\
\end{cases}
\end{eqnarray}
Then we have the following two cases:\\
$\rm{(i)}$  $\alpha=0$, namely, $\varphi_1(a)=1$, 
then {$\{\lambda_1, \varphi_1(a)\}$} uniquely determines $\omega, v$. By~\eqref{varphi1}, one has $x_1=a$. Hence $$\omega= \frac{1}{\lambda_1}\delta_a, \ v=0.$$ \\
$\rm{(ii)}$ $\alpha\neq0$, namely, $\varphi_1(a)=|\varphi_1(a)|<1$, then two $\omega, v$'s yield the given interior spectral data {$\{\lambda_1, \varphi_1(a)\}$}. From \eqref{varphi1}, one obtains
\begin{equation}
x_1=a+2\ln \varphi_1(a) \ {\rm{or}} \  x_1=a-2\ln \varphi_1(a). \notag
\end{equation} Hence we can conclude that
$$\omega=\frac{1}{\lambda_1}\delta_{a+2\ln \varphi_1(a)}, \ v=0,$$ or
$$\omega=\frac{1}{\lambda_1}\delta_{a-2\ln \varphi_1(a)}, \ v=0.$$
\end{example}

{	For any $N\in \mathbb{N}$, let $\mathcal{T}$ be the set of  real numbers $\{\lambda_i, \alpha_i\}_{i=1}^N$ with $\lambda_1<\lambda_2<\cdots<\lambda_N$
	satisfying
	\begin{equation} \label{sumalpha}
\sum_{i=1}^{N}\alpha^2_i\le1
	\end{equation}
	and 
	 \begin{align} \nonumber
	\lambda_{j_0}<0<\lambda_{j_0+1},\ \alpha_{j_0}>0,\ \alpha_{j_0+1}>0,\label{lambdaj0}
	\end{align}
	for some $j_0\in \{0,\cdots, N\}$,
	 where $j_0=0$ means that $\lambda_1>0, \alpha_1>0$ and $j_0=N$ means that $\lambda_N<0, \alpha_N>0$.}
	 
 We next discuss the existence of the interior spectral problem.

\begin{lemma}  \label{theorem3.8}
	Assume that
		\begin{equation}
	\alpha_i\ne0, i=1,\cdots,N.  \nonumber
	\end{equation}
	Then for given $a$, $\{\lambda_i, \alpha_i\}_{i=1}^N$ is the interior spectral data for some $\omega,v$ at $a$ if and only if $\{\lambda_i, \alpha_i\}_{i=1}^N\in \mathcal{T}$.

\end{lemma}
\begin{proof}
The necessity is obtained from Lemma \ref{lemma3.1} and Remark \ref{remark5.11}.  Following the similar steps in the proof of Theorem \ref{theorem3.3}, we can obtain {the} sufficiency.
\end{proof}

In the following, we require that if a certain symbol $\gamma$ denotes an object related to $H(\omega, v)$, then the corresponding symbol $\tilde{\gamma}$ with tilde denotes the analogous object related to $H(\tilde{\omega}$, $\tilde{v})$.
\begin{theorem}  \label{corollary 3.10}
For given $a$, suppose that $\{\lambda_i, \varphi_i(a) \}_{i=1}^N$ is the interior spectral data for some $\omega,v$ with
\begin{equation}
\varphi_i(a)\neq0 , i=1,\cdots, N. \notag  \\
\end{equation}
$\rm{(i)}$ If
\begin{equation}
\alpha=\beta=0,   \nonumber
\end{equation}
 then there are precisely $2^{N-2}$ $\tilde{\omega}, \tilde{v}$'s,  so that
\begin{equation}
\tilde{\lambda}_i=\lambda_i,|\tilde{\varphi}_i(a)|=|\varphi_i(a)| , i=1,\cdots,N. \notag
 \end{equation}
 $\rm{(ii)}$ If
\begin{equation}
\alpha=0,\ \beta\neq0,  \nonumber
\end{equation}
 then there are precisely $2^{N-1}$ $\tilde{\omega}, \tilde{v}$'s, so that
\begin{equation}
\tilde{\lambda}_i=\lambda_i,|\tilde{\varphi}_i(a)|=|\varphi_i(a)| , i=1,\cdots,N. \notag
 \end{equation}
 $\rm{(iii)}$ If
\begin{equation}
\alpha\neq0,   \nonumber
\end{equation}
 then there are precisely $2^N$ $ \tilde{\omega}, \tilde{v}$'s,  so that
\begin{equation}
\tilde{\lambda}_i=\lambda_i,|\tilde{\varphi}_i(a)|=|\varphi_i(a)| , i=1,\cdots,N. \notag
 \end{equation}
\end{theorem}
\begin{proof}
	We first show that ${\rm{(iii)}}$ holds. We split the proof into  two cases.\\
	Case 1: $\lambda_1>0$ or $\lambda_N<0$.
	 We only consider the case that $\lambda_1>0$. By Remark \ref{remark5.11},  we have that $\tilde{\varphi}_1(a)>0$. From Lemma \ref{theorem3.8}, for any $i\in \{2,\cdots,N\}$,  the sign of $\tilde{\varphi_i}(a)$ is not determined.   {Then} we  can choose the signs of $\tilde{\varphi}_1(a),\cdots,\tilde{\varphi}_N(a)$
	in $2^{N-1}$ ways. Every such choice yields two $\tilde{\omega},\tilde{v}$'s by (i) of Theorem \ref{theorem3.7}. Hence, we determine $2^{N}$ $\tilde{\omega},\tilde{v}$'s. \\
Case 2:  For some $j_0\in\{1,\cdots,N-1\}$, we have
	\begin{equation}
	\lambda_{j_0}<0<\lambda_{j_0+1}.\notag
	\end{equation}
By Remark \ref{remark5.11},  we have $\tilde{\varphi}_{j_0}(a), \tilde{\varphi}_{j_0+1}(a)>0$. From Lemma \ref{theorem3.8}, for any $i\in\{1,\cdots,j_0-1, j_0+2,\cdots,N\}$, the sign of $\tilde{\varphi}_{i}(a)$ is not determined.  Then we can choose the signs of $\tilde{\varphi}_1(a), \cdots,  \tilde{\varphi}_N(a)$
in $2^{N-2}$ ways. Every such choice yields four $\tilde{\omega},\tilde{v}$'s by (ii) of Theorem \ref{theorem3.7}. Hence, we determine $2^{N}$ $\tilde{\omega},\tilde{v}$'s.

Note that in both cases, we have (iii) holds. Applying the above
method with {slight modification},  we know that (i) and (ii) hold.
	\end{proof}


\section{Interior Inverse Problem II}
In this section\red{,} we consider the interior inverse problem with the case that $\varphi_i(a)=0$ for some $i$. We need the following lemmas.

\begin{lemma} \label{lemma3.11}
	For any $a\in \mathbb{R}$, if $\varphi_{i_0}(a)=0$ for some $i_0\in \{2,\cdots, N-1\}$,
	then $$\varphi_{i_0-1}(a)\varphi_{i_0+1}(a)<0.$$
\end{lemma}
\begin{proof}
 If $\varphi_{i_0}(a)=0$, by \eqref{xishu} and \eqref{eq2.6}, we have
  \begin{equation} \label{zero0}
  \phi_+(\lambda_{i_0},a)=\phi_-(\lambda_{i_0},a)=0.
  \end{equation}
  Hence, by \eqref{eq3.4},  one has that
  \begin{align} \label{zero}
  -\frac{1}{M_+(\lambda_{i_0},a)+M_-(\lambda_{i_0},a)}=0.
  \end{align}
We first show that $$\varphi_{i_0-1}(a)\ne0, \ \varphi_{i_0+1}(a)\ne0. $$
 Otherwise, for example, assume that $\varphi_{i_0-1}(a)=0$. Then 
 \begin{align} \label{zero1}
 -\frac{1}{M_+(\lambda_{i_0-1},a)+M_-(\lambda_{i_0-1},a)}=0.
 \end{align}
  Since $-1/ (M_+(z,a)+M_-(z,a))$
  is a Herglotz-Nevanlinna function, by \eqref{zero} and \eqref{zero1}, the function  $-1/ (M_+(z,a)+M_-(z,a))$ has a pole $\nu$ in $(\lambda_{i_0-1},\lambda_{i_0})$.
This contradicts \eqref{eq3.2}.

Since
$\varphi_{i_0-1}(a)\ne0$ and $ \varphi_{i_0+1}(a)\ne0$,
applying  \eqref{eq3.2} and \eqref{zero}, $M_+(z,a)+M_-(z,a)$ has  exactly one pole $\lambda_{i_0}$ in the interval $[\lambda_{i_0-1},\lambda_{i_0+1}]$. Then by \eqref{zero0}, we know that  $M_+(z,a)$ has exactly one pole $\lambda_{i_0}$ in
   $[\lambda_{i_0-1},\lambda_{i_0+1}]$. From  \eqref{eq212},
   $z\phi_+(z,a)$ has exactly one zero $\lambda_{i_0}$   in $[\lambda_{i_0-1},\lambda_{i_0+1}]$. Therefore,
    we have
   \begin{equation} \label{11}
\lambda_{i_0-1}\phi_+(\lambda_{i_0-1},a)\lambda_{i_0+1}\phi_+(\lambda_{i_0+1},a)<0.
   \end{equation}
    Note that  Theorem \ref{theorem5.9} implies
   $\lambda_{i_0-1}\lambda_{i_0+1}>0$.
   Then one obtains  $$\phi_+(\lambda_{i_0-1},a)\phi_+(\lambda_{i_0+1},a)<0.$$ By \eqref{eq2.6}, one has that $\varphi_{i_0-1}(a)\varphi_{i_0+1}(a)<0$.
\end{proof}
	
Suppose that there are $k$ many $\varphi_i(a)$ such that $\varphi_i(a)=0$,
	and $\varphi_1(a)\ne0$,
	$\varphi_N(a)\ne0$. For example, assume that  $\varphi_{m_1}(a)=\cdots=\varphi_{m_k}(a)=0$.
	By Lemma \ref{lemma3.11}, for any $j=1,\cdots,k$, we obtain that
$M_+(z,a)+M_-(z,a)$ has exactly  one pole $\lambda_{m_j}$ in $(\lambda_{m_j-1},\lambda_{m_j+1})$. For any $i\notin \{m_1-1,m_1, \cdots,m_k-1,m_k\}$, $M_+(z,a)+M_-(z,a)$ has exactly one pole in $(\lambda_{i},\lambda_{i+1})$.

\begin{lemma} \label{lemma3.13}
$\rm{(i)}$ Suppose that  \begin{equation}
\alpha=\beta=0. \nonumber
\end{equation}
Then   $\varphi_1(a)\ne0$ and $ \varphi_N(a)\ne0$.\\
$\rm{(ii)}$ Suppose that  $$\alpha=0, \ \beta\ne0.$$
Then $\varphi_1(a), \varphi_N(a)$ cannot both be zero.  If moreover, $\lambda_1>0$ or $\lambda_N<0$, then we have that
 $\varphi_1(a)\ne0$ and $ \varphi_N(a)\ne0$.
\end{lemma}

\begin{proof}
	Assume that there are $k$ many $\varphi_i(a)$ such that $\varphi_i(a)=0$.
Let us show (i) first.  Otherwise, for example, assume that $\varphi_1(a)=0$. From the proof of Lemma \ref{lemma3.11},
we know that
\begin{align} \label{zero2}
-\frac{1}{M_+(\lambda_{1},a)+M_-(\lambda_{1},a)}=0.
\end{align}
Since $\alpha=\beta=0$, by Lemma 3.1,  $-1/\left(M_+(z,a)+M_-(z,a)\right)$ has $N-k$ poles and $N-k-1$ zeros. We emphasize that $-1/({M_+(z,a)}+{M_-(z,a)})$ has one more pole than zeros. Combining with  \eqref{zero2}, one concludes that
 $-1/(M_+(z,a)+M_-(z,a))$ has one pole in $(-\infty, \lambda_1)$. This contradicts \eqref{eq3.2}.
	
	Now we are going to prove (ii). Otherwise, assume that
	$\varphi_1(a)=\varphi_N(a)=0$.  Then one has that
	\begin{align} \label{zero3}
	-\frac{1}{M_+(\lambda_{1},a)+M_-(\lambda_{1},a)}=-\frac{1}{M_+(\lambda_{N},a)+M_-(\lambda_{N},a)}=0.
	\end{align}
	Since $\alpha=0, \beta\ne0$, by Lemma 3.1,  we know that $-1/({M_+(z,a)}+{M_-(z,a)})$ has the same {{number}} of zeros and poles. Combining with \eqref{zero3}, $-1/(M_+(z,a)+M_-(z,a))$ has one pole in $(-\infty, \lambda_1) \cup (\lambda_N, +\infty) $. This contradicts \eqref{eq3.2}.
	
	It remains to establish that if moreover, $\lambda_1>0$ or $\lambda_N<0$, then
	$\varphi_1(a)\ne0, \varphi_N(a)\ne0$. We only consider the case that $\lambda_1>0$. By Theorem \ref{theorem5.9}, we know that $\varphi_1(a)\ne0$. We show that $ \varphi_N(a)\ne0$. Otherwise, we obtain that
	\begin{align} \label{zero4}
	-\frac{1}{M_+(\lambda_{N},a)+M_-(\lambda_{N},a)}=0.
	\end{align}
Since $\varphi_1(a)\ne0$, by \eqref{eq3.2},  $\lambda_1$ is a pole of $-1/({M_+(z,a)}+{M_-(z,a)})$. Using \eqref{eq3.2} and \eqref{eq3.4}, the smallest zero  and pole  of $-1/({M_+(z,a)}+{M_-(z,a)})$
 are $\mu_1=0$ and $\lambda_1$,  respectively, which
 satisfy
\begin{equation}\label{d}
\mu_1=0<\lambda_1.
\end{equation}
Note that $-1/({M_+(z,a)}+{M_-(z,a)})$ has the same {{number}} of zeros and poles. Then by \eqref{zero4} and \eqref{d}, $-1/(M_+(z,a)+M_-(z,a))$ has one pole in $(\lambda_N, +\infty)$. This contradicts \eqref{eq3.2}.
\end{proof}

{Let $\mathcal{F}$ be the collection of spectral data of inverse spectral family, which yield the given interior spectral data $\{\lambda_i, \varphi_i(a) \}_{i=1}^N$. The following theorems  tell us some properties of $\mathcal{F}$ including connectedness. Compared with the results in Section 3, the connectedness of $\mathcal{F}$ also depends on whether $\varphi_1(a)$ and $\varphi_N(a)$ are zero.  }

\begin{theorem} \label{theorem 3.14}
For given $a$, suppose that $\{\lambda_i, \varphi_i(a) \}_{i=1}^N$
is the interior spectral data for some $\omega, v$, and
\begin{equation}
\alpha=\beta=0. \nonumber
\end{equation}
If there are $k$ many $\varphi_i(a)$ such that $\varphi_i(a)=0$, then {infinitely many  $\omega, v$'s} yield the given interior spectral data $\{\lambda_i, \varphi_i(a) \}_{i=1}^N$.
$\mathcal{F}$ is a connected manifold of dimension $k$.
\end{theorem}

\begin{proof}
Since $\alpha=\beta=0$, by \eqref{G} and Lemma 3.1, we know that
\begin{equation}
M_+(z,a)+M_-(z,a)=\gamma z+\zeta+\sum_{i=1}^{N-k-1}\frac{\beta_i}{\mu_i-z},
\end{equation}
where $\{\mu_i\}_{i=1}^{N-k-1}$ are the poles of ${M_+(z,a)}+{M_-(z,a)}$.
Denote the zeros of  ${M_+(z,a)}+{M_-(z,a)}$ by $\{\lambda_{l_i}\}_{i=1}^{N-k}$.  Clearly, $\{\mu_i\}_{i=1}^{N-k-1}$ and $\{\lambda_{l_i}\}_{i=1}^{N-k}$ are interlaced:
\begin{equation}
\lambda_{l_1}<\mu_1<\lambda_{l_2}<\cdots<\lambda_{l_{j_1}}<\mu_{j_1}=0<\lambda_{l_{j_1+1}}<\cdots<\lambda_{l_{N-k}}. \label{lambda111}
\end{equation}

We divide the $N-k-1$ poles of $M_+(z,a)+M_-(z,a)$ into four parts:  \\
(1)    0; \\
(2)   $A$ is the set of  zeros of both functions $\phi_+(z,a)$ and $\phi_-(z,a)$.  We mention that $A$ coincides with the set of $\lambda_i$ with $\varphi_i(a)=0$; \\
(3)   $B$ is the set of zeros of $\phi_+(z,a)$ alone; \\
(4)   $C$ is the set of zeros of $\phi_-(z,a)$ alone.

Following the {similar} steps in the proof of Theorem \ref{theorem3.3},
by \eqref{lambda111}, the sets $B$ and $C$ are uniquely determined by the signs of $\varphi_i(a) , i=1,\cdots,N$. Then we make the following choice:
for any $\mu_i\in A$, let $\beta_i^{(1)},\beta_i^{(2)}>0$ be arbitrary with
 \begin{equation}
 \beta_i=\beta_i^{(1)}+\beta_i^{(2)},
 \end{equation}
 and take
 \begin{align}
M_+(z,a)&=\gamma z+ \zeta-\frac{1}{2z}+\sum_{\mu_i\in B} \frac{\beta_i}{\mu_i-z}+ \sum_{\mu_i\in A}\frac{\beta_{i}^{(1)}}{\mu_{i}-z},  \label{m+za} \\
M_-(z,a)&=-\frac{1}{2z}+\sum_{\mu_i\in C} \frac{\beta_i}{\mu_i-z}+\sum_{\mu_i\in A}\frac{\beta_{i}^{(2)}}{\mu_{i}-z}.
\end{align}
Every such choice yields an acceptable $\omega$, $v$  (see Theorem \ref{theorem3.18} for the acceptance of $\omega, v$).

Notice that by \eqref{w(0)} and \eqref{phi_+(0,a)}, one has 
\begin{align}
W(z)=\prod_{\lambda_i\in \sigma(H)}(1-\frac{z}{\lambda_i}),
\end{align}
\begin{equation} \label{phipmza}
\phi_+(z,a)= e^{-\frac{a}{2}} \prod_{\mu_i\in A \cup B}(1-\frac{z}{\mu_i}), \phi_-(z,a)= e^{\frac{a}{2}} \prod_{v_i\in A \cup C}(1-\frac{z}{v_i}).
\end{equation}
Therefore,  $W(z)$ and $\phi_{\pm}(z,a)$  are uniquely determined by $\sigma(H)=\{\lambda_i\}_{i=1}^N$ and the sets $A, B$ and $C$.
For any $j$  with $\varphi_{j}(a)\ne0$, by \eqref{xishu}, \eqref{eq2.5} and   \eqref{phipmza}, we have
\begin{equation} \label{lambdajgammalambdaj}
{\lambda_j}\gamma_{\lambda_j}^2=e^{-a}\prod_{\lambda_i\in \sigma(H) \setminus \{\lambda_j\}}(1-\frac{\lambda_j}{\lambda_i})
\prod_{\mu_i\in B }(1-\frac{\lambda_j}{\mu_i})
\prod_{v_i\in C}(1-\frac{\lambda_j}{v_i})^{-1},
\end{equation}
which is uniquely determined by $\sigma(H)$ and the sets $ B, C$.

For any $j$ with $\varphi_j(a)=0$, by \eqref{eq3.4}, one has that
\begin{equation} \label{beta}
{\beta_j}={\rm{Res}}_{z=\lambda_j}\frac{W(z)}{z\phi_+(z,a)\phi_-(z,a)}=\frac{\dot{W}(\lambda_j)}{\lambda_j\dot{\phi}_+(\lambda_j,a)\dot{\phi}_-(\lambda_j,a)}>0.
\end{equation}
Hence, one obtains that
\begin{align}\label{sgn1}
\sgn \Big({\frac{\dot{W}(\lambda_j)}{\dot{\phi}_+(\lambda_j,a)}}\Big)
=\sgn \big(\lambda_j\dot{\phi}_-(\lambda_j,a)\big).
\end{align}
Letting $x=a$  in \eqref{Wzdefine}, and then differentiating with respect to $z$ at $z=\lambda_j$, one has
\begin{align}\label{dotphi}
\dot{\phi}_+(\lambda_j,a)\phi_-'(\lambda_j,a)
-\phi_+'(\lambda_j,a)\dot{\phi}_-(\lambda_j,a)=\dot{W}(\lambda_j).
\end{align}
By \eqref{xishu}, we know
\begin{align} \label{xishu2}
\phi_-'(\lambda_j,a)=c_{\lambda_j}\phi_+'(\lambda_j,a).
\end{align}
Substituting \eqref{xishu2} into \eqref{dotphi}, one can obtain
\begin{align} \label{111}
\phi_+'(\lambda_j,a)=\frac{\dot{W}(\lambda_j)}
{c_{\lambda_j}\dot{\phi}_+(\lambda_j,a)-\dot{\phi}_-(\lambda_j,a)}.
\end{align}
From \eqref{eq212}, \eqref{m+za} and \eqref{111}, we have
\begin{align}
-\beta_j^{(1)}=\frac{\phi_+'(\lambda_j,a)}{\lambda_j\dot{\phi}_+(\lambda_j,a)}=-\frac{\dot{W}(\lambda_j)}{\dot{\phi}_+(\lambda_j,a)}\frac{1}{\frac{\lambda_j^2\dot{W}(\lambda_j)\dot{\phi}_+(\lambda_j,a)}{\lambda_j\gamma_{\lambda_j}^2}+\lambda_j\dot{\phi}_-(\lambda_j,a)},  \nonumber
\end{align}
which implies
\begin{align} \label{beta1}
 \lambda_j\gamma_{\lambda_j}^2=\frac{\lambda_j^2\dot{W}(\lambda_j)\dot{\phi}_+(\lambda_j,a)}{\frac{\dot{W}(\lambda_j)}{\beta_j^{(1)}\dot{\phi}_+(\lambda_j,a)}-\lambda_j\dot{\phi}_-(\lambda_j,a)}.
\end{align}
Since we can choose any $\beta_j^{(1)}\in (0,\beta_j)$, by \eqref{beta}, \eqref{sgn1} and \eqref{beta1}, we conclude that
$\lambda_j\gamma_{\lambda_j}^2\in(0,+\infty)$ can be arbitrary.

Hence, we have that
\begin{align}
\mathcal{F}&=\{(\lambda_1,\cdots, \lambda_N, \lambda_1\gamma_1^2,\cdots, \lambda_N\gamma_N^2)\in \mathbb{R}^{2N}|\lambda_j \ {\rm{is \ given \ for \ any}}  \ j, \nonumber\\
& \lambda_j\gamma_j^2 \ { \rm{is \ given \ by \ \eqref{lambdajgammalambdaj}}} \ {\rm{for}} \ \varphi_j(a)\ne0, \lambda_j\gamma_j^2>0  \ {\rm{for}} \ \varphi_j(a)=0 \},  \nonumber
\end{align}
which is a connected manifold of dimension $k$.
\end{proof}

\begin{remark}
Note that there are $k$ many eigenvalues $\lambda_j \in \{\mu_i\}_{i=1}^{N-k-1} \setminus \{0\} $  with  $\varphi_j(a)=0$.
Hence we have that $N-k-2\ge k$, which means that
\begin{align}
N-2k-2\ge 0. \nonumber
\end{align}
\end{remark}

{From the proof of Thoerem \ref{theorem 3.14}, we can compute $\mathcal{F}$ explicitly, which is a manifold of dimension $k$. Then by the reconstruction method in \cite[Theorem 4.3]{eck2}, we can obtain all the $w, v$ which yield the given interior spectral data $\{\lambda_i, \varphi_i(a) \}_{i=1}^N$. } 

\begin{theorem}
For given $a$, suppose that $\{\lambda_i, \varphi_i(a) \}_{i=1}^N$
is the interior spectral data for some $\omega, v$, and
\begin{equation}
\alpha=0,  \ \beta\ne 0. \nonumber
\end{equation}
If there are $k$ many $\varphi_i(a)$ such that $\varphi_i(a)=0$, then  {infinitely many  $\omega, v$'s} yield the given interior spectral data $\{\lambda_i, \varphi_i(a) \}_{i=1}^N$.
 We have \\
 $\rm{(i)}$  if $\lambda_1>0$ or $\lambda_N<0$, then $\mathcal{F}$ is a connected manifold of dimension $k$;\\
 $\rm{(ii)}$ if for some {$j_0\in \{1,\cdots, N-1\}$},
 \begin{equation}
 \lambda_{j_0}<0<\lambda_{j_0+1},   \notag
 \end{equation}
 and  $\varphi_1(a)=0$ or $ \varphi_N(a)=0$, then $\mathcal{F}$ is a connected manifold of dimension $k$;\\
 $\rm{(iii)}$  if for some {$j_0\in \{1,\cdots, N-1\}$},
 \begin{equation}
 \lambda_{j_0}<0<\lambda_{j_0+1},   \notag
 \end{equation}
 and
 \begin{equation}
   \varphi_1(a)\ne0, \ \varphi_N(a)\ne0,
   \end{equation}
  then $\mathcal{F}$ is a collection of 2 disjoint manifolds, each of dimension $k$.
\end{theorem}
\begin{proof}
	Since $\alpha=0,\beta\ne0$, by \eqref{G} and Lemma 3.1, we know that
	\begin{equation}
	M_+(z,a)+M_-(z,a)=\zeta+\sum_{i=1}^{N-k}\frac{\beta_i}{\mu_i-z},
	\end{equation}
	where $\{\mu_i\}_{i=1}^{N-k }$ are the poles of ${M_+(z,a)}+{M_-(z,a)}$. Note that there are $k$ many eigenvalues $\lambda_j \in \{\mu_i\}_{i=1}^{N-k} \setminus \{0\} $  with  $\varphi_j(a)=0$. Hence we have that $N-2k-1\ge 0$.
	
Let us show ${\rm{(i)}}$ first. We only consider the case that $\lambda_1>0$. Denote the zeros of  ${M_+(z,a)}+{M_-(z,a)}$ by $\{\lambda_{l_i}\}_{i=1}^{N-k}$. Then  one has	
	\begin{equation}
\mu_1=0<\lambda_{l_1}<\mu_2<\lambda_{l_2}<\cdots<\lambda_{l_{N-k-1}}<\mu_{N-k}<\lambda_{l_{N-k}}. \label{lambda112}
	\end{equation}
We have that $\lambda_{l_{N-k}}=\lambda_N$. Then the sets $B$ and $C$,  which are defined in the proof of Theorem \ref{theorem 3.14}, are uniquely determined by the signs of $\varphi_i(a),i=1,\cdots,N$.

	Now we are going to prove  (ii) and (iii). Without loss of generality, assume that
\begin{equation}
\lambda_{l_1}<\mu_1<\lambda_{l_2}<\cdots<\lambda_{l_{j_1}}<\mu_{j_1}=0<\lambda_{l_{j_1+1}}<\cdots<\lambda_{l_{N-k}}<\mu_{N-k}. \label{lambda113}
\end{equation}	

	If $\varphi_N(a)=0$, then $\mu_{N-k}=\lambda_N$.  Namely, $\mu_{N-k}\in A$. Therefore,  the sets $B$ and $C$ are uniquely determined by the signs of $\varphi_i(a), i=1,\cdots,N$. This yields (ii).
	
If $\varphi_N(a)\ne0$, then $\mu_{N-k}>\lambda_N=\lambda_{l_{N-k}}$. Namely, $\mu_{N-k}\notin A$.  Since  $\mu_{N-k}\in B \ {\rm{or}} \  C$, then we can obtain two disjoint manifolds of dimension $k$. This yields (iii).
\end{proof}

\begin{theorem} \label{theorem3.16}
For given $a$, suppose that $\{\lambda_i, \varphi_i(a) \}_{i=1}^N$
is the interior spectral data for some $\omega, v$, and
\begin{equation}
\alpha\ne 0. \nonumber
\end{equation}
Assume that $\lambda_1>0$ (or $\lambda_N<0$).
If there are $k$ many $\varphi_i(a)$ such that $\varphi_i(a)=0$, then  {infinitely many  $\omega, v$'s} yield the given interior spectral data $\{\lambda_i, \varphi_i(a) \}_{i=1}^N$.
 We have \\
 $\rm{(i)}$ if $\varphi_N(a) = 0$  (or $\varphi_1(a) = 0$),  then $\mathcal{F}$ is a connected manifold of dimension $k$; \\
 $\rm{(ii)}$ if $\varphi_N(a)\ne 0$  (or $\varphi_1(a)\ne 0$),  then $\mathcal{F}$ is a collection of 2 disjoint manifolds, each of dimension $k$.
\end{theorem}
\begin{proof}
Since $\alpha\neq0$, by \eqref{G} and Lemma 3.1, there exist $\beta_i,\mu_i, i=0,\cdots,N$, (~determined by $\{\lambda_i, \varphi_i(a) \}_{i=1}^{N-k}$ ) so that
\begin{equation}
M_+(z,a)+M_-(z,a)=\sum_{i=0}^{N-k} \frac{\beta_i}{\mu_i-z},      \label{N+11}
\end{equation}
where $\{\mu_i\}_{i=0}^{N}$ are the poles of ${M_+(z,a)}+{M_-(z,a)}$.  Note that there are $k$ many eigenvalues $\lambda_j \in \{\mu_i\}_{i=1}^{N-k-1} \setminus \{0\} $  with  $\varphi_j(a)=0$. Hence we have $N-2k\ge 0$.

 Denote the zeros of  ${M_+(z,a)}+{M_-(z,a)}$ by $\{\lambda_{l_i}\}_{i=1}^{N-k}$. Then $\{\mu_i\}_{i=0}^{N-k}$ and $\{\lambda_{l_i}\}_{i=1}^{N-k}$ are interlaced:
	\begin{equation}
\mu_0=0<\lambda_{l_1}<\mu_1<\lambda_{l_2}<\cdots<\lambda_{l_{N-k-1}}<\mu_{N-k-1}<\lambda_{l_{N-k}}<\mu_{N-k}. \label{lambda114}
\end{equation}	

If $\varphi_N(a)=0$, then $\mu_{N-k}=\lambda_N$.  Namely, $\mu_{N-k}\in A$.  Here the set $A$ is defined in the proof of Theorem \ref{theorem 3.14}.  Then by \eqref{lambda114}, the sets $B$ and $C$ are uniquely determined by the signs of $\varphi_i(a), i=1,\cdots,N$. This yields (i).

If $\varphi_N(a)\ne0$, then $\mu_{N-k}>\lambda_N=\lambda_{l_{N-k}}$. Namely, $\mu_{N-k} \notin A$. Since  $\mu_{N-k}\in B \ {\rm{or}} \  C$,  we can obtain two disjoint  manifolds of dimension $k$. This yields (ii).
\end{proof}
\begin{theorem} \label{theorem3.17}
	For given $a$, suppose that $\{\lambda_i, \varphi_i(a) \}_{i=1}^N$
	is the interior spectral data for some $\omega, v$, and
	\begin{equation}
	\alpha\ne 0. \nonumber
	\end{equation}
	Assume that for some {$j_0\in \{1,\cdots, N-1\}$}, we have
	\begin{equation}
	\lambda_{j_0}<0<\lambda_{j_0+1}.   \label{eq1}
	\end{equation}
	If there are $k$ many $\varphi_i(a)$ such that $\varphi_i(a)=0$, then  {infinitely many  $\omega, v$'s} yield the given interior spectral data $\{\lambda_i, \varphi_i(a) \}_{i=1}^N$. We have\\
	$\rm{(i)}$ if  $\varphi_1(a)=0$ and $ \varphi_N(a)=0$, then $\mathcal{F}$ is a connected manifold of dimension $k$;\\
	$\rm{(ii)}$ if  $\varphi_1(a)=0$ or $ \varphi_N(a)=0$ (cannot both be zero), then $\mathcal{F}$ a collection of 2 disjoint manifolds, each of dimension $k$; \\
	$\rm{(iii)}$ if  $\varphi_1(a)\ne 0$ and $ \varphi_N(a)\ne 0$, then $\mathcal{F}$ is a collection of 4 disjoint manifolds, each of dimension $k$.
\end{theorem}
\begin{proof}
 By \eqref{N+11} and \eqref{eq1}, we have that
\begin{equation}
\mu_0<\lambda_{l_1}<\cdots<\lambda_{l_{j_1}}<\mu_{j_1}=0<\lambda_{l_{j_1+1}}<\cdots<\lambda_{l_{N-k}}<\mu_{N-k}. \label{lambda12}
\end{equation}	

If $\varphi_1(a)=0$ and $ \varphi_N(a)=0$, then
\begin{equation} \nonumber
 \mu_0=\lambda_1, \ \mu_{N-k}=\lambda_N.
 \end{equation}
 Namely, $\mu_0, \mu_{N-k}\in A$. Therefore,  the sets $B$ and $C$ are uniquely determined by the signs of $\varphi_i(a), i=1,\cdots,N$. This yields (i).
 
Let us show (ii). We only consider the case that $\varphi_1(a)=0, \varphi_N(a)\ne 0$. Then
\begin{equation} \nonumber
\mu_0=\lambda_1, \ \mu_{N-k}>\lambda_N=\lambda_{l_{N-k}}.
\end{equation}
 Namely,  $\mu_0\in A$ and $\mu_{N-k}\notin A$.
 Since  $\mu_{N-k}\in B \ {\rm{or}} \  C$, we can obtain two disjoint  manifolds. This yields (ii).

If $\varphi_1(a)\ne0$ and $ \varphi_N(a)\ne0$, then
\begin{equation} \nonumber
\mu_0<\lambda_1=\lambda_{l_1}, \ \mu_{N-k}>\lambda_N=\lambda_{l_{N-k}}.
\end{equation} Namely,
$\mu_0,  \mu_{N-k}\notin A$. Since  $\mu_0, \mu_{N-k}\in B \ {\rm{or}} \  C$, {we can obtain four disjoint  manifolds.} This yields (iii).
\end{proof}

Now we discuss the existence of the interior inverse problem.

 \begin{theorem}\label{theorem3.18}
For given $a$, $\{\lambda_i,\alpha_i\}_{i=1}^N\subset\mathbb{R} $ is the interior spectral data for some $\omega,v$ at $a$, if and only if \\
$\rm{(i)}$ for any $j$ with $\alpha_j=0$, $\lambda_j$ is the zero of the function
\begin{align}  \label{Gzz}
G(z):=(1-\sum_{i=1}^N \alpha_i^2)z-\sum_{i=1}^N \lambda_i\alpha_i^2+\sum_{i=1}^N \frac{\lambda_i^2\alpha_i^2} {\lambda_i-z};
\end{align}
$\rm{(ii)}$ if $\alpha_j=0$ for some $j\ne1, N$,  then we have $\alpha_{j-1}\alpha_{j+1}<0$;\\
$\rm{(iii)} $  $\{\lambda_i,\alpha_i\}_{i=1}^N \in \mathcal{T} $,  {where $\mathcal{T} $ is defined in Section 3.}
\end{theorem}

\begin{proof}
The necessity is obtained from Lemma \ref{lemma3.1}, Lemma \ref{lemma3.11} and Remark \ref{theorem5.9}.

Let us show the sufficiency. Assume that there are $k$ many $\alpha_i$ such that $\alpha_i=0$. If $k=0$, the sufficiency holds by Lemma \ref{theorem3.8}. We consider the case that $k>0$.
 From \eqref{sumalpha},  we know that $G(z)$ is a Herglotz-Nevanlinna function. Let
\begin{equation}
M(z):=-\frac{1}{G(z)}=M_+(z)+M_-(z), \label{definitionmz}
\end{equation}
where $M_+(z)$ and $M_-(z)$ will be determined later.  Then there exist $\gamma\ge0$, $\zeta$, $\beta_i,\mu_i, i=1,\cdots,N(k)$, so that
	\begin{equation}
	M(z)= \gamma z+\zeta+\sum_{i=1}^{N(k)}\frac{\beta_i}{\mu_i-z}.
	\end{equation}
 Let $A$ be the set of $\lambda_j$ with $\alpha_j=0$. By (i), we know that $A \subseteq \{\mu_i\}_{i=1}^{N(k)}$.

We make the following choice. 
 For  $\mu_i\ne0\in (\lambda_j, \lambda_{j+1})$ for some $j$, let $\mu_i$ be the pole of $M_{\pm}(z)$ if
 \begin{equation} \label{cc}
  \pm \alpha_{j}\alpha_{j+ 1}<0.
 \end{equation}
  For $\mu_i\in (-\infty,  \lambda_1)\cup (\lambda_N, +\infty)$ with $\mu_i\ne0$ (if exist), we can choose $\mu_i$ to be the pole of $M_+(z)$. Denote the above poles of $M_+(z)$ and $M_-(z)$ by $B$ and $C$, respectively.
 For any $ \mu_i \in  A$, let $\beta_i^{(1)},\beta_i^{(2)}>0$ satisfy
 \begin{equation}
 \beta_i=\beta_i^{(1)}+\beta_i^{(2)},
 \end{equation}
 and take
\begin{align}
M_+(z)&=\gamma z+ \zeta-\frac{1}{2z}+\sum_{\mu_i\in B} \frac{\beta_i}{\mu_i-z}+ \sum_{\mu_i\in A}\frac{\beta_{i}^{(1)}}{\mu_{i}-z}, \label{m+}\\
M_-(z)&=-\frac{1}{2z}+\sum_{\mu_i\in C} \frac{\beta_i}{\mu_i-z}+\sum_{\mu_i\in A}\frac{\beta_{i}^{(2)}}{\mu_{i}-z}.  \label{m-}
\end{align}
  Applying Lemma \ref{lemma2.6} and Lemma \ref{lemma2.7}, $M_+(z)$ and $M_-(z)$ uniquely determine $\omega|_{[a,+\infty)}, v|_{[a,+\infty)}$ and $\omega|_{(-\infty,a)}, v|_{(-\infty,a)}$, respectively. Hence $\omega$ and $v$ are uniquely determined.

	{ It remains to establish that the $\omega, v$ determined above has the interior spectral data $\{\lambda_i, \alpha_i \}_{i=1}^N$.} Note that the Weyl-Titchmarsh functions $M_{\pm}(z,a)$ of $H_{\pm}(\omega, v,a)$ satisfy
	\begin{equation}\label{m}
	M_{\pm}(z,a)=M_{\pm}(z). 
	\end{equation}
	  If $\lambda_j\in \{\lambda_i\}_{i=1}^{N} \setminus A$,  by \eqref{definitionmz}, \eqref{m+}, \eqref{m-} and \eqref{m}, $\lambda_j$ is a pole of $-1/(M_+(z,a)+M_-(z,a))$. Then by Lemma \ref{lemma3.1}, $\lambda_j$ is an eigenvalue of $H(\omega, v)$. If $\lambda_j \in A$, by \eqref{eq212} and \eqref{M-za}, we obtain $$\phi_+(\lambda_j,a)= \phi_-(\lambda_j,a)=0.$$ Then from \eqref{Wzdefine}, we know that $W(\lambda_j)=0$,  which implies that $\lambda_j$ is also an eigenvalue of $H(\omega, v)$.
Since $H(\omega,v)$ has no other eigenvalues, one concludes that $\{\lambda_i\}_{i=1}^{N}$ is the spectrum of $H(\omega,v)$.

	Now we show that for any $i= 1,\cdots, N$, one has  $\varphi_{i}(a)=\alpha_{i}$. We consider the case that $\lambda_1>0$.
  By \eqref{Gzz}, \eqref{definitionmz}, \eqref{m} and Lemma \ref{lemma3.1}, one has  $$\abs{\varphi_i(a)}=|\alpha_i|, i=1,\cdots,N.$$
  Therefore, it suffices to show that for any $i=1,\cdots, N$, we have
  \begin{align}\label{sgnequal1}
  {\rm{sgn}}(\varphi_i(a))={\rm{sgn}}(\alpha_i).
  \end{align} 
  Applying  Remark \ref{remark5.11}, one has that \eqref{sgnequal1} {holds} for $i=1$.
   If  $\mu_i\in (\lambda_j, \lambda_{j+1})$ for some $j$, we know that 
	 $\mu_{i}$ is a pole of $M_{\pm}(z,a)$ if and only if $\pm\varphi_{j}(a)\varphi_{j+1}(a)<0$. Combining with Lemma \ref{lemma3.11}, the signs of ${\varphi}_{i}(a)$ for $i\neq 1$ are uniquely determined by ${\rm{sgn}} ({\varphi}_{1}(a))$.  Therefore, by (ii) and the definition of $\mathcal{T}$, we can obtain that \eqref{sgnequal1} holds for any $i=1,\cdots,N$.

	   Following {the} similar steps,
	we can obtain that  $\varphi_{i}(a)=\alpha_{i}$ for any $i= 1,\cdots,N,$
	under the condition that $\lambda_N<0$ or 
	\begin{equation}
	\lambda_{j_0}<0<\lambda_{j_0+1}   \notag
	\end{equation}
	for some $j_0$.
	Therefore, we conclude that $\{\lambda_i,\alpha_i\}_{i=1}^N $ is the interior spectral data for some $\omega,v$ at $a$.
\end{proof}

\begin{theorem} \label{corollary 3.15}
For given $a$, suppose that $\{\lambda_i, \varphi_i(a) \}_{i=1}^N$
is the interior spectral data for some $\omega, v$,
 and there are $k$ many $\varphi_i(a)$ such that $\varphi_i(a)=0$. Then there are infinitely many $\tilde{\omega}, \tilde{v}$, such that
\begin{equation}
\tilde{\lambda}_i=\lambda_i,|\tilde{\varphi}_i(a)|=|\varphi_i(a)| , i=1,\cdots,N. \notag
 \end{equation}
We have  \\
 $\rm{(i)}$ if \begin{equation}
\alpha=\beta=0,  \nonumber
\end{equation}
  then the  spectral data of the inverse spectral family is a collection of $2^{N-2k-2}$ disjoint manifolds, each of  dimension $k$;\\
  $\rm{(ii)}$ if \begin{equation}
\alpha=0, \ \beta\ne0,  \nonumber
\end{equation}
  then the  spectral data of the inverse spectral family is a collection of $2^{N-2k-1}$ disjoint manifolds, each of dimension $k$;\\
  $\rm{(iii)}$ if \begin{equation}
\alpha\ne 0,     \nonumber
\end{equation}
  then the  spectral data of the inverse spectral family is a collection of $2^{N-2k}$ disjoint manifolds, each of dimension $k$.
\end{theorem}
\begin{proof}
	For example, assume that $\varphi_{m_1}(a)= \cdots=\varphi_{m_k}(a)=0$. We first show that ${\rm{(iii)}}$ holds. \\
	Case 1: $\lambda_1>0$ or $\lambda_N<0$. We only consider the case that $\lambda_1>0$. By Remark \ref{remark5.11},  we have that $\tilde{\varphi}_1(a)>0$. If
	$ \varphi_N(a)\ne0$, namely $m_k<N$,
	 by Theorem \ref{theorem3.18}, for any $$i\in \{2,\cdots,N\}\setminus \{{m_1},{m_{1}+1},\cdots, {m_k},{m_{k}+1}\},$$ the sign of $\tilde{\varphi}_i(a)$ is not determined.
	We can choose the signs of $\tilde{\varphi}_1(a),\cdots, \tilde{\varphi}_N(a)$ in $2^{N-2k-1}$ ways.
	Every such choice yields 2 disjoint manifolds by (ii) of Theorem \ref{theorem3.16}.  Hence, we can determine $2^{N-2k}$ disjoint manifolds.
	
	 If	$ \varphi_N(a)=0$, namely $m_k=N$, by Theorem \ref{theorem3.18}, for any
	 $$i \in\{2,\cdots,N\}\setminus \{{m_1},{m_{1}+1},\cdots, {m_{k-1}},{m_{k-1}+1},{m_{k}} \},$$ the sign of  $\tilde{\varphi}_i(a)$ is not determined.
	  We can choose the signs of $\tilde{\varphi}_1(a),\cdots, \tilde{\varphi}_N(a)$ in $2^{N-2k}$ ways.
	  Every such choice yields a connected manifold by (i) of Theorem \ref{theorem3.16}.  Hence, we can determine $2^{N-2k}$ disjoint manifolds.\\
Case 2:  For some $j_0\in\{1,\cdots,N-1\}$, we have
	\begin{equation}
	\lambda_{j_0}<0<\lambda_{j_0+1}.   \notag
	\end{equation}
	By Remark \ref{remark5.11},  we know that  $\tilde{\varphi}_{j_0}(a), \tilde{\varphi}_{j_0+1}(a)>0$.
	
If $\varphi_1(a)= \varphi_N(a)=0$,
	by  Theorem \ref{theorem3.18}, for any
	$$i\in \{1,\cdots,N\}\setminus \{j_0,j_0+1,{m_1}, {m_{2}}, {m_{2}+1},\cdots, {m_{k-1}},{m_{k-1}+1},{m_k} \},$$
	the sign of $\tilde{\varphi}_{i}(a)$ is not determined.
 We can choose the signs of $\tilde{\varphi}_1(a),\cdots, \tilde{\varphi}_N(a)$ in $2^{N-2k}$ ways.
  Every such choice yields a connected manifold by (i) of Theorem \ref{theorem3.16}.  Hence, we can determine $2^{N-2k}$ disjoint manifolds.
	
	  Similarly, if $\varphi_1(a)\ne0, \varphi_N(a)=0$ or $\varphi_1(a)\ne0, \varphi_N(a)\ne0$, we can also determine
	 $2^{N-2k}$ disjoint manifolds.
	
	  Note that in both cases, the  spectral data of the inverse spectral family is a collection of $2^{N-2k}$ disjoint manifolds, each of  dimension $k$.
	 Applying the above
	 method with slight modification, we can show (i) and (ii).	
\end{proof}

\section{The Camassa-Holm Equation}

Let $(u,\mu)$ be a  multipeakon global conservative solution of the CH equation (see \cite{eck2}). Here $\mu$ is a non-negative Borel measure with {the} absolutely continuous part determined by
$u$ via
\begin{equation}
\mu_{ac}(B,t)= \int_B |u(x,t)|^2+|u_x(x,t)|^2 dx,\  t\in\mathbb{R},   \label{eq4.1}
\end{equation}
for each Borel set $B\in \mathcal{B}(\mathbb{R})$.   Define
\begin{align}
 \omega(x,t)&=u(x,t)-u_{xx}(x,t)=\sum_{i=1}^{n(t)} \omega_i(t)\delta_{x_i(t)},   \label{eq4.3} \\
 v(B,t)&=\mu(B,t)-\int_B |u(x,t)|^2+|u_x(x,t)|^2 dx,\  B\in \mathcal{B}(\mathbb{R}),    \label{eq4.4}
\end{align}%
for each $t\in \mathbb{R}$.

Consider the family (parametrized by time $t\in \mathbb{R})$ of spectral problems
\begin{equation}
-f''(x)+\frac{1}{4}f(x)=z \omega(x, t) f(x)+z^2 v(x,t) f(x),\  x\in \mathbb{R},         \label{eq4.5}
\end{equation}
with $\omega(x,t)$ and $v(x,t)$ defined by \eqref{eq4.3} and $\eqref{eq4.4}$, respectively.  We will denote all spectral quantities associated with this spectral problem as in the preceding sections but with an additional time parameter. In particular, the spectrum of \eqref{eq4.5} will be denoted by $\sigma(H(t))$. By\cite[Theorem 5.1]{eck2}, the pair $(u,\mu)$ is a global conservative multipeakon solution of the CH equation if and only if the problems \eqref{eq4.5} are isospectral with
\begin{equation}
\gamma_{\lambda_i}^2(t)=e^{-\frac{t-t_0}{2\lambda_i}}\gamma_{\lambda_i}^2(t_0),\  t\in \mathbb{R},\  \lambda_i\in \sigma(H(t_0)).    \label{eq4.6}
\end{equation}
Here $\gamma_{\lambda_i}^2(t)$ is defined by \eqref{eq2.4} with an additional time parameter.

Now in a position to state the trace formula.
\begin{theorem} \label{theorem 4.1}
 Suppose that $(u,\mu)$ is a global conservative multipeakon solution of the CH equation and let $\{\lambda_i\}_{i=1}^N$ be eigenvalues of \eqref{eq4.5} with $\omega(x,t)$ and $v(x,t)$  defined by \eqref{eq4.3} and \eqref{eq4.4}, respectively. For a fixed $t\in \mathbb{R}$, let $\{\varphi_i(x,t)\}_{i=1}^N$ be the corresponding normalized eigenfunctions given by \eqref{eq2.6}. Then one has
\begin{equation}
u(x,t)=\frac{1}{2}\sum_{i=1}^N \frac{|\varphi_i(x,t)|^2}{\lambda_i}.   \label{eq4.7}
\end{equation}
\end{theorem}
\begin{proof}
 For a fixed $t\in{\mathbb{R}}$, by Lemma \ref{lemma3.1},  for any  $x\in \mathbb{R}$, one has
 \begin{equation}
-\frac{1}{M_+(z,x)+M_-(z,x)}= \alpha z+\beta+G(z,x). \label{eq4.8}
\end{equation}
Here
\begin{align}
\alpha=1-\sum_{i=1}^N |\varphi_i(x)|^2,   \beta=-\sum_{i=1}^N \lambda_i|\varphi_i(x)|^2.  \nonumber
\end{align}
  Here we omit the $t$-dependance in functions $G(z,x), M(z,x), \varphi_i(x)$ for simplicity. From \cite[p. 901]{eck2}, we know that
  \begin{align}  \nonumber
  \dot{\phi}_{\pm}(0,x)&=e^{\mp\frac{x}{2}}\int_{I_{\pm(x)}}(e^{-|x-s|}-1)d\omega(s),\\
   \dot{W}(0)&=-\int_{\mathbb{R}} \omega(s)ds,  \nonumber
  \end{align}
 where $I_+(x)=[x,+\infty)$ and $I_-(x)=(-\infty,x)$.  Combining with  \eqref{w(0)} and \eqref{phi_+(0,a)}, one can obtain that
\begin{align}
\phi_{\pm}(z,x)&=e^{\mp\frac{x}{2}}+ze^{\mp\frac{x}{2}}\int_{I_\pm(x)}(e^{-|x-s|}-1)d\omega(s)+O(z^2),\ z\rightarrow0,\notag \\
W(z)&=1-z\int_{\mathbb{R}} \omega(s)ds+O(z^2),\ z\rightarrow0. \notag
\end{align}
Then from \eqref{eq3.4}, we have
\begin{equation}
-\frac{1}{M_+(z,x)+M_-(z,x)}=z+z^2\int_{\mathbb{R}}e^{-|x-s|}d\omega(s)+O(z^3), \ z\rightarrow0.   \label{eq4.9}
\end{equation}
By  \eqref{TaylorofGza} and \eqref{eq4.9}, one obtains
\begin{equation}
\int_{\mathbb{R}}e^{-|x-s|}d\omega(s)=\sum_{i=1}^N \frac{|\varphi_i(x)|^2}{\lambda_i}.      \label{Oz2}
\end{equation}
Then from  \eqref{eq4.3}, we obtain that
\begin{equation}
u(x)=\frac{1}{2}\int_{\mathbb{R}}e^{-|x-s|}d\omega(s)=\frac{1}{2}\sum_{i=1}^N \frac{|\varphi_i(x)|^2}{\lambda_i}. \nonumber
\end{equation}
\end{proof}

The Cauchy problem for the CH equation is solving $u(x,t)$ from $u(x,0)$.  Namely,  solve $u(x,t)$ by $x_i(0), \omega_i(0), i=1,\cdots,N$
in \eqref{eq1.3}. Similarly, \eqref{eq4.7} enlightens us to consider the following question: \\
\textbf{For given $x_0\in \mathbb{R}$, can we solve $u(x,t)$ from $\lambda_i, |\varphi_i(x_0,0)|, i=1,\cdots, N$?}
We have the following theorem.

\begin{theorem}
  Suppose that $(u,\mu)$ is a global conservative multipeakon solution of the CH equation.\\
 $\rm{(i)}$ For a fixed $x_0\in\mathbb{R}$, if $$\varphi_i(x_0,0)\neq0, i=1,\cdots, N, $$ then there exist finite global conservative multipeakon solutions   $(\tilde{u},\tilde{\mu})$, so that
\begin{equation}
\tilde{\lambda}_i=\lambda_i , |\tilde{\varphi}_i(x_0,0)|= |\varphi_i(x_0,0)|,  \ \  i=1,\cdots,N. \notag
\end{equation}
Namely,
\begin{equation}
\tilde{u}(x_0,0)=\frac{1}{2}\sum_{i=1}^N\frac{|\varphi_i(x_0,0)|^2}{\lambda_i}. \notag
\end{equation}
 $\rm{(ii)}$ For a fixed $x_0$, if $\varphi_i(x_0,0)=0$ for some $i$, then there exist infinite many global conservative multipeakon solutions $(\tilde{u},\tilde{\mu})$, so that
\begin{equation}
\tilde{\lambda}_i=\lambda_i, |\tilde{\varphi}_i(x_0,0)|= |\varphi_i(x_0,0)|,  \ \  i=1,\cdots,N. \notag
\end{equation}
Namely,
\begin{equation}
\tilde{u}(x_0,0)=\frac{1}{2}\sum_{i=1}^N\frac{|\varphi_i(x_0,0)|^2}{\lambda_i}. \notag
\end{equation}
\end{theorem}
\begin{proof}
	By Theorem \ref{corollary 3.10} and Theorem \ref{theorem 4.1}, we have \rm{(i)}. 	Applying Theorem \ref{corollary 3.15} and Theorem \ref{theorem 4.1}, we obtain (ii).
\end{proof}

Denote by
\begin{equation}
\lambda^{\star}= \min\{|\lambda_i|: i=1,\cdots,N\}.   \label{definitionA}
\end{equation}
Combining  Lemma \ref{lemma3.1}, Lemma \ref{lemma3.2} and Theorem \ref{theorem 4.1}, we obtain the supremum of $|u(x,t)|$. This supremum only depends on $\lambda^{\star}$.

\begin{theorem} Assume that $(u,\mu)$ is a global conservative multipeakon solution of the CH equation. Let $\{\lambda_i\}_{i=1}^N$ be eigenvalues of \eqref{eq4.5} with $\omega(x,t)$ and $v(x,t)$ defined by \eqref{eq4.3} and \eqref{eq4.4}, respectively. Then we have
\begin{equation}
\sup \{ |u(x,t)|: x, t \in \mathbb{R} \}= \frac{1}{2\lambda^{\star}}. \label{sup|u(x,t)|}
\end{equation}
In particular, $|u(x,t)|$ obtains the supremum at some $(x,t)\in \mathbb{R}^2$ if and only if $N=1$.
\end{theorem}

\begin{proof}
By \eqref{eq4.7}, \eqref{definitionA} and Lemma \ref{lemma3.1}, we have
\begin{align}
|u(x,t)| &\le \frac{1}{2}\sum_{i=1}^N\frac{|\varphi_i(x,t)|^2}{|\lambda_i|} \\ \notag
& \le\frac{1}{2\lambda^{\star}}\sum_{i=1}^N|\varphi_i(x,t)|^2 \le \frac{1}{2\lambda^{\star}}.    \notag
\end{align}
Hence, $1/2\lambda^{\star}$ is a upper bound for $|u(x,t)|$. Combining with $(6.5)$ in \cite{bea}, one obtains that \eqref{sup|u(x,t)|} holds  (in the normalization used in this paper).

If $N=1$, for any $t\in\mathbb{R}$, $|\varphi_1(x_1(t)),t)|=1$ by Lemma \ref{lemma3.2}. Then by \eqref{eq4.7}, one has $$|u(x_1(t),t)|=\frac{1}{2\lambda^{\star}}.$$
Therefore for any $t\in\mathbb{R}$, $|u(x,t)|$ obtains the supremum at the point $(x_1(t),t)$.

 If $N=2$,  for any $t\in \mathbb{R}$, if $x\neq x_j(t), j=1, 2,$ one has $$\sum_{i=1}^2 |\varphi_i(x,t )|^2\ < 1$$ by Lemma \ref{lemma3.2}. Then we obtain that for
 $x\neq x_j(t), j=1, 2,$
\begin{align}
|u(x,t)|\le\frac{1}{2\lambda^{\star}}\sum_{i=1}^2|\varphi_i(x,t)|^2 < \frac{1}{2\lambda^{\star}}.  \label{|u(x,t)|1}
\end{align}
 Note that
 \begin{equation}
 \varphi_i(x_j(t), t)\neq0, \quad    i,j=1,2.  \label{neq0}
 \end{equation}
  Otherwise, for example, assume that
  \begin{equation}
  \varphi_1(x_1(t), t)=0.  \label{varphi1(x)}
  \end{equation}
  	 Since for $x\le x_1(t)$, $\varphi_1(x,t)=Ae^{\frac{x}{2}}$ for some constant $A$.  By \eqref{varphi1(x)}, one has $A=0$. Then from \eqref{lianxu}, \eqref{tiaoyue},   for any  $x\in \mathbb{R}$,  one obtains $\varphi_1(x,t)=0$. This is in  contradiction {with} the fact that $\varphi_1(x,t)$ is an eigenfunction. By \eqref{neq0}, for $j=1,2$,
\begin{align}
|u(x_j(t),t)|<\frac{1}{2\lambda^{\star}}\sum_{i=1}^2|\varphi_i(x_j(t),t)|^2 = \frac{1}{2\lambda^{\star}}. \label{|u(x,t)|2}
\end{align}
Hence from \eqref{|u(x,t)|1} and \eqref{|u(x,t)|2}, we conclude that if $N=2$, $|u(x,t)|$ cannot obtain the supremum at some $(x,t)\in \mathbb{R}^2$.

{If $N \geq 3$, we need to consider the following two cases. \\
Case 1. $\lambda_1>0$ or $\lambda_N<0$.
For any $t\in \mathbb{R}$, by Corollary \ref{corollary5.10},  we have
\begin{align}
|u(x,t)|<\frac{1}{2\lambda^{\star}}\sum_{i=1}^N|\varphi_i(x,t)|^2 \leq \frac{1}{2\lambda^{\star}}. \label{ccc}
\end{align}
Case 2. For some $j_0\in\{1,\cdots,N-1\}$, we have
\begin{equation}
\lambda_{j_0}<0<\lambda_{j_0+1}.   \label{c}
\end{equation}
 Then by Theorem \ref{theorem5.9} one has that
\begin{equation} \notag
\Big|\frac{|\varphi_{j_0}(x,t)|^2}{2\lambda_{j_0}}+\frac{|\varphi_{j_0+1}(x,t)|^2}{2\lambda_{j_0+1}}\Big|<\frac{|\varphi_{j_0}(x,t)|^2+|\varphi_{j_0+1}(x,t)|^2}{2\lambda^{\star}}.
\end{equation}
This implies that \eqref{ccc} holds.
Therefore, if $N\ge 3$, then $|u(x,t)|$ cannot obtain the supremum at some $(x,t)\in \mathbb{R}^2$.} 

 From all the above,  we conclude that $|u(x,t)|$ obtains the supremum at some $(x,t)\in \mathbb{R}^2$ if and only if $N=1$.
\end{proof}
If $N=1$, corresponding to the single peakon, $|u(x,t)|$ obtains the supremum.  If $u$ is a multipeakon solution with $N\geq 2$,   $|u(x,t)|\le 1/2\lambda^{\star}$ means other peakons have a negative effect on this highest (deepest) peakon which has a height (depth) of $1/2\lambda^{\star}$. $|u(x,t)|$ cannot obtain the supremum  implies that  as $t$ becomes large, the negative effect becomes smaller but never disappear.

\section{Oscillation of eigenfunctions}
In this section, we give a proof of the oscillation theorem for problem \eqref{eq2.1}. The approach we are taking in this section follows in part the one in \cite{ma}.  But we take advantage of the property of Herglotz-Nevanlinna function  and greatly simplify the proof.  

{In this part, if an eigenvalue $\lambda$ is positive (or negative), we denote it by $\lambda^+$ (or $\lambda^-$).}


For $x_i=1,\cdots, n-1$, one has
\begin{align}
\phi_{+}'(z,x_i+)&= \frac{1}{2\sinh\left(\frac{ x_{i+1}-x_i}{2}\right)}\phi_{+}(z,x_{i+1})-\frac{1}{2}\coth\left(\frac{x_{i+1}-x_i}{2}\right)\phi_{+}(z, x_i). \label{phi}
\end{align}
Indeed,  for any $i=1,\cdots,n-1$, there exists $A, B\in \mathbb{R}$, so that
 \begin{equation} \label{phi(z,x)}
  \phi_+(z,x)=Ae^{\frac{x}{2}}+Be^{-\frac{x}{2}},\  x_i\le x\le x_{i+1}.
  \end{equation}
   Then we have
   \begin{equation}  \nonumber
   \left\{
   \begin{aligned}
   Ae^{\frac{x_i}{2}}+Be^{-\frac{x_i}{2}}&=\phi_+(z,x_i),\\
   \frac{1}{2}Ae^{\frac{x_i}{2}}-\frac{1}{2}Be^{-\frac{x_i}{2}}&=\phi'_+(z,x_i+). \\
   \end{aligned}
   \right.
   \end{equation}
   As a consequence,  one obtains
   \begin{equation} \label{AB}
   A=\frac{1}{2}e^{-\frac{x_i}{2}}(\phi_+(z,x_i)+2\phi'_+(z,x_i+)),
  B=\frac{1}{2}e^{\frac{x_i}{2}}(\phi_+(z,x_i)-2\phi'_+(z,x_i+)).
   \end{equation}
   Substituting \eqref{AB} into \eqref{phi(z,x)} and let $x=x_{i+1}, $ we obtain
   \begin{align} \label{phi(z,x_{i+1}}
   \phi_+(z,x_{i+1})=&\frac{1}{2}e^{\frac{x_{i+1}-x_i}{2}}(\phi_+(z,x_i)+2\phi'_+(z,x_i+))\nonumber\\
   &+\frac{1}{2}e^{-\frac{x_{i+1}-x_i}{2}}(\phi_+(z,x_i)-2\phi'_+(z,x_i+)).
   \end{align}
  Therefore, \eqref{phi} holds.
 Similarly, for $i=2,\cdots,n$,
\begin{align}
 \phi_{+}'(z, x_i-)&=-\frac{1}{2\sinh\left(\frac{
 x_i-x_{i-1}}{2}\right)}\phi_{+}(z, x_{i-1})+\frac{1}{2}\coth\left(\frac{x_{i}-x_{i-1}}{2}\right)\phi_{+}(z, x_{i}).  \label{phi+}
\end{align}

Denote by
 \begin{equation}
 \phi_i(z)=\phi_{+}(z,x_{n-i}), i=1,\cdots,n, \label{definitionvarphii}
 \end{equation}
and let
 \begin{align}
 a_i&=\frac{1}{2\sinh\left(\frac{x_{n-i+1}-x_{n-i}}{2}\right)}   \label{ai}, i=1,\cdots, n-1, \\
 b_i&=\frac{1}{2}\left(\coth\left(\frac{x_{n-i+1}-x_{n-i}}{2}\right)
 +\coth\left(\frac{x_{n-i}-x_{n-i-1}}{2}\right)\right), i=0,\cdots, n-1. \label{bi}
 \end{align}
 Here  $x_0:=-\infty$ and $x_{n+1}:=+\infty$.
{By \eqref{eq2.222}   and \eqref{definitionvarphii}, we have
 \begin{align}
 \phi_0(z)&=e^{-\frac{x_n}{2}}.  \label{phi1}
 \end{align}
Letting $i=n$ in \eqref{phi+}, then substituting equations \eqref{phi+} and  $$\phi_+'(z,\lambda_n+)=-\frac{1}{2}e^{-\frac{x_n}{2}}$$ into
 \eqref{tiaoyue}, we can obtain
 \begin{align}
 -a_1\phi_1(z)&+b_0\phi_0(z)=(z^2v_{n}+z\omega_n)\phi_0(z).   \label{phi2}
 \end{align}
 Substituting \eqref{phi} and \eqref{phi+} into \eqref{tiaoyue} with $i=n-j+1$,  one has that for $ j=2,\cdots, n-1$,
 \begin{align}
 -a_{j}\phi_{j}(z)+b_{j-1}\phi_{j-1}(z)-a_{j-1}\phi_{j-2}(z)=(z^2v_{n-j+1}+z\omega_{n-j+1})\phi_{j-1}(z).  \label{phi3}
 \end{align}
 Therefore, one can obtain $\phi_j(z),j=0,\cdots, n-1,$ from \eqref{phi1}, \eqref{phi2} and \eqref{phi3}.
 
Recall that 
 $\lambda$ is an eigenvalue of $H$ if and only if there exists $A\in \mathbb{R}$, so that  $\phi_+(\lambda,x)=Ae^{\frac{x}{2}}$ for $x\le x_1$. Namely, 
\begin{equation} \label{z}
\phi_+'(\lambda,x_1-)=\frac{1}{2}\phi_+(\lambda,x_1).
\end{equation} 
Substituting \eqref{z} and \eqref{phi} into \eqref{tiaoyue} with $i=1$, we can obtain
\begin{equation} \nonumber
b_{n-1}\phi_{n-1}(\lambda)-a_{n-1}\phi_{n-2}(\lambda)=(\lambda^2v_{1}+\lambda\omega_{1})\phi_{n-1}(\lambda).   
\end{equation}
Therefore}, $\lambda$ is an eigenvalue of $H$ if and only if $\lambda$ is a root  of
\begin{equation}
b_{n-1}\phi_{n-1}(z)-a_{n-1}\phi_{n-2}(z)=(z^2v_{1}+z\omega_{1})\phi_{n-1}(z).   \label{b_{n-1}}
\end{equation}

 {For any $i\in\{1,\cdots, n\}$, define 
 	\begin{align} \nonumber
 	i_v&= \#\{j\in\{1,\cdots, i\}|v_{n-j+1}\ne0 \},\\
 	i_+&=\#\{j\in\{1,\cdots, i\}| v_{n-j+1}=0, \omega_{n-j+1}>0\},\\
 	i_-&=\#\{j\in\{1,\cdots, i\}| v_{n-j+1}=0, \omega_{n-j+1}<0\}.
 	\end{align}
 	For any $j$ with $v_{n-j+1}\ne 0$ and $j_v=k$,
 define }
 \begin{equation} \label{y}
 \varphi_k(z):=z\sqrt{v_{n-j+1}}\phi_{j-1}(z).
 \end{equation}
   Denote by
 \begin{equation}\nonumber
 y=(\phi_0,\cdots,\phi_{n-1},\varphi_1,\cdots,\varphi_{n_v})^T.
 \end{equation}
 Then by \eqref{phi2}, \eqref{phi3}, \eqref{b_{n-1}} and \eqref{y}, we can rewrite the spectral problem $H$ in the matrix form
 \begin{equation} \label{jy}
  Jy=zDy.
 \end{equation}
Here
\begin{align} \nonumber
J=\left(
\begin{array}{cc}
J_{11} & J_{12} \\
J_{21} & J_{22} \\
\end{array}
\right),
\end{align}
$J_{11}$ is an $n\times n$ matrix with
\begin{align}\nonumber
J_{11}=\left(
\begin{array}{ccccccccc}
b_0    & -a_1    & 0    & \cdots & 0          \\
-a_1    & b_1    & -a_2  & \cdots & 0       \\
0      & -a_2    & b_2  & \cdots & 0        \\
\cdots &\cdots  &\cdots& \cdots & \cdots    \\
0      & 0      & 0    & -a_{n-1}& b_{n-1}
\end{array}
\right),
\end{align}
$J_{22}$ is an $n_v\times n_v $ identity matrix, $J_{12}$  is an $n\times n_v$ zero matrix, $J_{21}$  is an $n_v\times n$ zero matrix. $D$ is a symmetric matrix
with
\begin{align}  \label{definitionDD}
D=\left(
\begin{array}{cc}
D_{11} & D_{12} \\
D_{21} & D_{22} \\
\end{array}
\right),
\end{align}
 $D_{11}$ is an $n\times n$ matrix with
\begin{align} \nonumber
D_{11}=\left(
\begin{array}{ccccccccc}
\omega_n    & 0    & 0    & \cdots & 0          \\
0    & \omega_{n-1}   & 0  & \cdots & 0       \\
0      & 0    & \omega_{n-2}  & \cdots & 0        \\
\cdots &\cdots  &\cdots& \cdots & \cdots    \\
0      & 0      & 0    & 0 & \omega_{1}
\end{array}
\right),
\end{align}
$D_{22}$ is an $n_v\times n_v $ zero matrix, $D_{12}$  is an $n\times n_v$  matrix with entries
\begin{equation} \nonumber
d_{jk}(D_{12})=\left\{
\begin{aligned}
\sqrt{v_{n-j+1}}, \quad j_v=k,  \\
0, \quad \rm{otherwise}. \\
\end{aligned}
\right.
\end{equation}
Define
\begin{align} \nonumber
J_{i}=\left(
\begin{array}{cc}
J_{11,i} & J_{12,i} \\
J_{21,i} & J_{22,i} \\
\end{array}
\right),
\end{align}
where $J_{11,i}$ is an $i$-$th$ principal submatrix of $J_{11}$,
$J_{22,i}$ is an $i_v\times i_v $ identity matrix, $J_{12,i}$  is an $i\times i_v$ zero matrix, $J_{21,i}$  is an $i_v\times i$ zero matrix. Define the symmetric matrix
\begin{align}  \label{definitionD}
D_i=\left(
\begin{array}{cc}
D_{11,i} & D_{12,i} \\
D_{21,i} & D_{22,i} \\
\end{array}
\right),
\end{align} where $D_{11,i}$ is an $i$-$th$ principal submatrix of $D_{11}$,
$D_{22,i}$ is an $i_v\times i_v $ zero matrix, $D_{12,i}$  is an $i\times i_v$  matrix with entries
\begin{equation} \nonumber
d_{jk}(D_{12,i})=\left\{
\begin{aligned}
\sqrt{v_{n-j+1}}, \quad j_v=k,  \\
0, \quad \rm{otherwise}. \\
\end{aligned}
\right.
\end{equation}

Let $Q_i(z)=\det (J_i-zD_i)$ and suppose that
{\begin{equation}
Q_0(z)=1. \label{q0}
\end{equation}}
  Then from \eqref{jy}, after some calculations, one has
\begin{align}
Q_1(z)&=b_0-\omega_n z-v_n z^2;     \label{q1}\\
Q_{i}(z)&=(b_{i-1}-\omega_{n-i+1} z-v_{n-i+1} z^2)Q_{i-1}(z)-a_{i-1}^2Q_{i-2}(z), \  i=2,\cdots, n. \label{qi=1}
\end{align}

\begin{lemma}  \label{lemma5.1}	
	For any $i=1,\cdots, n$, $Q_{i-1}(z)$ and  $Q_{i}(z)$ have no common zeros.
\end{lemma}
\begin{proof}
	Otherwise, if there exists $z=\lambda_0$ so that $Q_{i-1}(\lambda_0)=0$ and  $Q_{i}(\lambda_0)=0$. By \eqref{qi=1}, one has  $Q_{i-2}(\lambda_0)=0$. Then we can obtain that $Q_{i-3}(\lambda_0)=\cdots=Q_0(\lambda_0)=0, $
	which contradicts $\eqref{q0}$.
	\end{proof}

\begin{lemma} \label{lemma5.3}
For any $i=1,\cdots, n$, the zeros of $Q_i(z)$ are real and simple. Moreover,  $Q_i(z)$ has $i_v+i_+$ positive zeros and $i_v+i_-$ negative zeros.
\end{lemma}
 \begin{proof}
     
For any $\xi=(\xi_1,\cdots, \xi_{i+i_v})\in\mathbb{R}^{i+i_v}$, we have that
 	\begin{align}
 	\xi J_i\xi^T=&(b_0-a_1)\xi_1^2+(b_{i-1}-a_{i-1})\xi_i^2+\sum_{j=2}^{i-1}(b_{j-1}-a_{j-1}-a_j)\xi_j^2  \nonumber\\
 	&+\sum_{j=1}^{i-1}a_j(\xi_{j+1}-\xi_j)^2+ \sum_{j=i+1}^{i+i_v} \xi_j^2.        \nonumber
 	\end{align}
 	By \eqref{ai} and \eqref{bi}, one obtains
 		 $b_0-a_1>0,\ b_{i-1}-a_{i-1}>0$ and
 \begin{equation} \nonumber
 b_{j-1}-a_{j-1}-a_j>0,\ j=2,\cdots, i-1.
 \end{equation}
 	Then one has that 
 	\begin{equation}
 	\xi J_i\xi^T\ge 0,
 	\end{equation}
 	and $\xi J_i\xi^T=0$ if and only if $\xi=0$. Therefore,  $J_i$ is positive definite.
 	
 	 For any $j\in \{1,\cdots, i\}$ with $v_{n-j+1}=0$,  by \eqref{definitionD}, $\omega_{n-j+1}$ is an eigenvalue of $D_i$. Obviously, for any $j\in \{1,\cdots, i\}$ with $v_{n-j+1}\ne0$, all the eigenvalues of the matrix
 	\begin{align}  \label{matrixn-j+1}
 	\left(
 	\begin{array}{cc}
 	\omega_{n-j+1} & \sqrt{v_{n-j+1} }\\
 	 \sqrt{v_{n-j+1}} &   0 \\
 	\end{array}
 	\right)
 	\end{align}
 are also eigenvalues of the matrix $D_i$.  By direct calculations,
 \begin{equation} \nonumber
 \lambda_{\pm}=\frac{\omega_{n-j+1}\pm\sqrt{\omega_{n-j+1}^2+4v_{n-j+1}}}{2}
 \end{equation}
are eigevalues of the matrix \eqref{matrixn-j+1}. Hence, $\lambda_{\pm}$ are also eigevalues of the matrix $D_i$ with $\lambda_{\pm}\gtrless 0$. By the definitions of $i_+, i_-$ and $i_v$, we conclude that the matrix $D_i$ has $i_v+i_+$ positive zeros and $i_v+i_-$ negative zeros.

Since $J_i$ is positive definite, $J_i$ has the form $B^2$, where the  matrix $B$ is symmetric and positive definite. Then the spectral problem 
\begin{equation}  \label{ji}
J_i \xi=zD_i \xi
\end{equation}
 is equivalent to $B^{-1}D_iB^{-1}w=z^{-1}w, w=B\xi$. {Note that the 
geometric multiplicity of each eigenvalue for problem \eqref{ji} is one, this also holds for the matrix $B^{-1}D_iB^{-1}$.}
Since  $B^{-1}D_iB^{-1}$ is a symmetric  matrix, the eigenvalues of $B^{-1}D_iB^{-1}$ are real and the algebraic multiplicity of each eigenvalue equals geometric multiplicity, which equals $1$. Therefore, we conclude that the zeros of $Q_i(z)$ are real and simple.  

Applying Sylvester's Law of Inertia, the matrix $B^{-1}D_iB^{-1}$ has the same number of positive and negative eigenvalues as $D_i$. Then problem \eqref{ji} has $i_v+i_+$ positive zeros and $i_v+i_-$ negative zeros.
\end{proof}

\begin{remark}
	\rm{(i)} If $v_{n-i}=0$, $\omega_{n-i}>0$, then $ Q_{i+1}(z)$ and $Q_{i}(z)$ has the same number of negative zeros, $ Q_{i+1}(z)$ has one more positive zero than $Q_{i}(z)$;\\
	\rm{(ii)} if $v_{n-i}=0$, $\omega_{n-i}<0$, then $ Q_{i+1}(z)$ and $Q_{i}(z)$ has the same number of positive zeros, and $Q_{i+1}(z)$ has one more negative zero than $Q_{i}(z)$;\\
	\rm{(iii)} if $v_{n-i}>0$, then $Q_{i+1}(z)$ has one more positive zero and negative zero than $Q_{i}(z)$.
\end{remark}

\begin{lemma}   \label{a}
	For any $i=1,\cdots, n-1,$ one has that
	\begin{equation} \label{qihephii}
	Q_{i}(z)=\left(e^{\frac{x_n}{2}}\prod_{j=1}^{i} a_j\right)\phi_{i}(z),
	\end{equation}
and
	\begin{equation}
	Q_n(z)=\left(e^{\frac{x_n}{2}-\frac{x_1}{2}}\prod_{j=1}^{n-1} a_j\right)W(z),  \label{wz}
	\end{equation}
	where $W(z)$ is defined by \eqref{Wzdefine}.
\end{lemma}
\begin{proof}
	We first show that \eqref{qihephii} holds. Let $$u_0=e^{\frac{x_n}{2}}\phi_{0}(z), u_{i}(z)= \left(e^{\frac{x_n}{2}}\prod_{j=1}^{i} a_j\right)\phi_{i}(z),\ {i=1,\cdots,n-1}.   $$
	Then
	\begin{align}
	u_0(z)&=1;  \nonumber \\
	u_1(z)&=b_0-\omega_n z-v_n z^2; \nonumber\\
	u_{i}(z)&=(b_{i-1}-\omega_{n-i+1} z-v_{n-i+1} z^2)u_{i-1}(z)-a_{i-1}^2u_{i-2}(z), \ i=2,\cdots,n-1.  \nonumber
	\end{align}
   {By \eqref{q0}, \eqref{q1} and \eqref{qi=1}, for any $i=1,\cdots, n-1,$ we have  $u_i(z)=Q_i(z)$.}
 
	Now we are going to prove that \eqref{wz} holds. By \eqref{eq2.5} and Lemma \ref{lemma5.3},  $W(z)$ and $Q_n(z)$ have the same zeros and  multiplicities. Then there exists $\gamma\in\mathbb{R}$, so that
	\begin{equation}
	Q_n(z)=\gamma W(z).
	\end{equation}
	From  \eqref{qi=1} and  \eqref{qihephii},
	\begin{align}
	Q_n(0)&=b_{n-1}Q_{n-1}(0)-a_{n-1}^2Q_{n-2}(0)   \nonumber \\
	&=e^{\frac{x_n}{2}}(b_{n-1}\phi_{n-1}(0)-a_{n-1}\phi_{n-2}(0))\prod_{j=1}^{n-1} a_j. \label{qn}
	\end{align}
	Substituting  \eqref{phi_+(0,a)}, \eqref{ai} and \eqref{bi} into \eqref{qn}, by some calculations, we have that
	\begin{equation}
	Q_n(0)=e^{\frac{x_n}{2}-\frac{x_1}{2}}\prod_{j=1}^{n-1} a_j. \label{w0}
	\end{equation}
	Combining \eqref{w0} with \eqref{w(0)}, we can obtain \eqref{wz}.
\end{proof}

Let $\mathcal{A}$ be the set of pairs $(A, B)$ of real polynomials satisfying the following conditions: \\
$\rm{(i)}$
$\rm{deg}(\emph{A}(\emph{z}))+1 \le deg (\emph{B}(\emph{z})) \le \rm{deg} (\emph{A}(\emph{z}))+2$;\\
$\rm{(ii)}$ the zeros of $A(z)$ and $B(z)$ are real and simple, and  $A(0)> 0, B(0)> 0$;  \\
$\rm{(iii)}$ the largest negative zero  $\lambda_1^-(A)$ and the smallest positive zero $\lambda_1^+(A)$ of $A(z)$,  the largest negative zero $\lambda_1^-(B)$ and the smallest positive zero $\lambda_1^+(B)$ of $B(z)$ satisfy
\begin{equation} \nonumber
\lambda_1^-(A)< \lambda_1^-(B)<\lambda_1^+(B)<\lambda_1^+(A);
\end{equation}
$\rm{(iv)}$ the positive zeros of $A(z)$ and $B(z)$ separate one another; the negative zeros of $A(z)$ and $B(z)$ separate one another.

\begin{lemma} \label{lemma5.6}
 For any $i=1,\cdots, n$,  we have that 
 \begin{equation} \label{e}
 (Q_{i-1}(z),Q_{i}(z))\in \mathcal{A}.
 \end{equation}
\end{lemma}
\begin{proof}
Let us  show that for any $i=1,\cdots,n$, the function
$-{Q_{i}(z)}/{zQ_{i-1}(z)}	$
is a Herglotz-Nevanlinna function by induction.
By \eqref{q0} and \eqref{q1}, we obtain that
\begin{equation} \nonumber
-\frac{Q_{1}(z)}{zQ_0(z)}=	v_n z+ \omega_n- \frac{b_0}{z}.
\end{equation}
Hence, 	$-Q_{1}(z)/ zQ_0(z)$ is  a Herglotz-Nevanlinna function.

 Assume that $-Q_{p-1}(z)/ zQ_{p-2}(z)$ is a Herglotz-Nevanlinna function. Then $zQ_{p-2}(z)/ Q_{p-1}(z)$ is also a Herglotz-Nevanlinna function.
Let $\lambda_0$ be the zero of $Q_{p-1}(z)$.  Then $\lambda_0\ne0 $ and
\begin{equation}  \label{reszQ_{i-1}(z)}
{\rm{Res}}_{z=\lambda_0} \frac{zQ_{p-2}(z)} {Q_{p-1}(z)}= \frac{\lambda_0Q_{p-2}(\lambda_0)}{\dot{Q}_{p-1}(\lambda_0)}<0.
\end{equation}
 By \eqref{qi=1}, one obtains that
\begin{equation} \nonumber
-\frac{Q_{p}(z)}{zQ_{p-1}(z)}=	v_{n-p+1} z+ \omega_{n-p+1}- \frac{b_{p-1}}{z}+\frac{a_{p-1}^2Q_{p-2}(z)}{zQ_{p-1}(z)}.
\end{equation}
Combining with \eqref{reszQ_{i-1}(z)}, we have that
\begin{equation}  \label{11111111}
{\rm{Res}}_{z=\lambda_0}- \frac{Q_{p}(z)} {zQ_{p-1}(z)}= \frac{a_{p-1}^2Q_{p-2}(\lambda_0)}{\lambda_0\dot{Q}_{p-1}(\lambda_0)}<0.
\end{equation}	
	Note that $z=0$ is also a pole of $-Q_{p}(z)/zQ_{p-1}(z)$.	
	From  \eqref{w(0)}, \eqref{phi_+(0,a)} and  Lemma \ref{a},  one has
	 \begin{align}
	 Q_i(0)>0, i=0,\cdots, n.  \label{qi0}
	 \end{align}
	Therefore,
\begin{equation}  \label{1111111}
{\rm{Res}}_{z=0}- \frac{Q_{p}(z)}{zQ_{p-1}(z)}= -\frac{Q_{p}(0)}{Q_{p-1}(0)}<0.
\end{equation}	
By \eqref{11111111}, \eqref{1111111}  and the fact that $v_{n-p+1}\ge 0$, we obtain that  $-Q_{p}(z)/zQ_{p-1}(z)$ is a Herglotz-Nevanlinna function. Hence,  for any $i=1,\cdots,n$, the function $-Q_{i}(z)/zQ_{i-1}(z)$ is a Herglotz-Nevanlinna function.

By Lemma \ref{lemma5.1}, for any $i=1,\cdots,n$, $zQ_{i-1}(z)$ and $Q_{i}(z)$ have no common zeros. Then the zeros of $zQ_{i-1}(z)$ and $Q_{i}(z)$ are simple and interlaced. Combining with \eqref{qi0}, we have  $(Q_{i-1}(z),Q_{i}(z))\in \mathcal{A}$.
\end{proof}

	 By Lemma \ref{lemma5.3}, we can denote the zeros of $Q_{i}(z)$ by
\begin{equation} \nonumber
\lambda_{i^v+i^-}^-(Q_i)<\cdots<
\lambda_{1}^-(Q_i)<0<\lambda_{1}^+(Q_i)<\cdots<\lambda_{i^v+i^+}^+(Q_i).
\end{equation}
 From Lemma \ref{lemma5.6}, we know \\
 (i) if 	$\lambda_{p_v+p_+}^+(Q_{p})<\lambda_{(p-1)_v+(p-1)_+}^+(Q_{p-1})$ for some $p$,
 then $ Q_{p}(z)$ and $Q_{p-1}(z)$ has the same number of positive zeros. Namely,
 \begin{equation}
 p_v+p_+=(p-1)_v+(p-1)_+;   \label{i0}  \\
 \end{equation}
 (ii) if 	$\lambda_{(p-1)_v+(p-1)_+}^+(Q_{p-1})<\lambda_{p_v+p_+}^+(Q_{p})$ for some $p$, then  $ Q_{p}(z)$ has one more positive zero than $Q_{p-1}(z)$. Namely,
 \begin{equation}
 p_v+p_+=(p-1)_v+(p-1)_++1.   \label{i01}  \\
 \end{equation}

Let $S_i(z)$ be the sign-changing number of the sequence $Q_0(z), \cdots$, $Q_{i}(z)$. Then we have the following lemma.

\begin{lemma}\label{lemma5.7}
	We have that
\begin{equation}   \label{Si}
S_i(z)=
\begin{cases}
0,  \quad &z\in [0,\lambda_{1}^+(Q_i)),   \\
j, \quad &z\in (\lambda_{j}^+(Q_i),\lambda_{j+1}^+(Q_{i})), 1 \le j\le i_v+i_+-1,\\
i_v+i_+, \quad & z\in (\lambda_{i_v+i_+}^+(Q_i),+\infty).
\end{cases}
\end{equation}
Here  if $i_v+i_+=0$, \eqref{Si} means that $S_i(z)=0$ for $z\in [0,\infty)$.
\end{lemma}
\begin{proof}
		We can proceed by induction. Obviously, $S_0(z), S_1(z)$ has the representation \eqref{Si}.

	Assume that for $ i \le p-1$,  \eqref{Si} holds. We show that for $i=p$, \eqref{Si} also holds.
	
		For $z\in [0,\lambda_{1}^+(Q_{p}))$, by Lemma \ref{lemma5.6}, one has   $$Q_i(z)>0, i=1,\cdots, p.$$ Namely,  $S_{p}(z)=0$.
		
		For $z\in (\lambda_{j}^+(Q_{p}),\lambda_{j+1}^+(Q_{p}))$ with $ 1 \le j\le p_v+p_+-1,$ using \eqref{qi0},  one has
		\begin{align}
		\sgn \ Q_{p}(z)&=(-1)^{j}.  \label{sgnq}
		\end{align}
		Notice that  by Lemma \ref{lemma5.6}, we have 
	$$\lambda_{j}^+(Q_{p})<\lambda_{j}^+(Q_{p-1})<\lambda_{j+1}^+(Q_{p}).$$	Then we consider the following three cases: \\
		{\bf{Case 1.}} $z\in (\lambda_{j}^+(Q_{p}),\lambda_{j}^+(Q_{p-1}))$.  From the induction hypothesis, we have
		\begin{equation}
		S_{p-1}(z)=j-1   \label{Si0}
		\end{equation}
		   and hence
		\begin{align}
	\sgn \ Q_{p-1}(z)&=(-1)^{j-1}. \label{sgn}	
	\end{align}
	By \eqref{sgnq}, \eqref{Si0} and \eqref{sgn}, one obtains  that $S_{p}(z)=j$.\\
	{\bf{Case 2.}} $z\in (\lambda_{j}^+(Q_{p-1}),\lambda_{j+1}^+(Q_{p}))$. By the induction hypothesis,  we obtain that $S_{p-1}(z)=j$ and hence $$\sgn \ Q_{p-1}(z)=(-1)^{j}.$$  Combining with \eqref{sgnq}, we have that $S_{p}(z)=j$. \\
	{\bf{Case 3.}} $z=\lambda_{j}^+(Q_{p-1})$. 
	 by \eqref{sgnq}, we obtain
	 \begin{equation}
\sgn \ Q_{p}(\lambda_{j}^+(Q_{p-1}))=(-1)^{j}. \label{sgnqi}
	  \end{equation}
Then from \eqref{qi=1} and the fact that
	 \begin{equation}
	  Q_{p-1}(\lambda_{j}^+(Q_{p-1}))=0, \label{qqqq}
	  \end{equation} one can conclude
	\begin{equation}
\sgn \ Q_{p-2}(\lambda_{j}^+(Q_{p-1}))=(-1)^{j-1}.
	\end{equation}
 By the continuity of   $Q_{p-2}(z)$, there exists $\varepsilon>0$,   for any $z\in (\lambda_{j}^+(Q_{p-1})-\varepsilon,\lambda_{j}^+(Q_{p-1})+\varepsilon)$, we have that
	\begin{equation}
	\sgn \ Q_{p-2}(z)=(-1)^{j-1}.  \label{sgnnnn}
\end{equation}
 By the induction hypothesis, for any $z\in (\lambda_{j}^+(Q_{p-1})-\varepsilon,\lambda_{j}^+(Q_{p-1})+\varepsilon)$, one has
 \begin{equation}
 S_{p-2}(z)= C, \notag
 \end{equation}
where $C=1$ or $-1$.   Note that for any $z\in (\lambda_{j}^+(Q_{p-1})-\varepsilon,\lambda_{j}^+(Q_{p-1}))$,
by \eqref{Si0}, \eqref{sgn} and \eqref{sgnnnn},
we deduce
\begin{equation}
S_{p-2}(z)=(-1)^{j-1}   \notag
\end{equation}
and hence $C=(-1)^{j-1}.$ Then  we can obtain that
	\begin{equation}
	S_{p-2}(\lambda_{j}^+(Q_{p-1}))=(-1)^{j-1}. \nonumber
	\end{equation}
	Combining with \eqref{sgnqi} and \eqref{qqqq},  we conclude that
	\begin{equation}
	S_{p}(\lambda_{j}^+(Q_{p-1}))=(-1)^{j}.  \nonumber
	\end{equation}

Hence, for $z\in(\lambda_{j}^+(Q_{p}),\lambda_{j+1}^+(Q_{p}))$ with $1 \le j\le p_v+p_+-1,$  one has that
	\begin{equation}
	S_{p}(z)=(-1)^{j}.  \nonumber
	\end{equation}

For $z\in (\lambda_{p_v+p_+}^+(Q_{p}),+\infty)$,
	using \eqref{qi0}, one can obtain
	\begin{equation}
	\sgn\ Q_{p}(z)=(-1)^{p_v+p_+}. \label{qqq}
	\end{equation}
If  $\lambda_{(p-1)_v+(p-1)_+}^+(Q_{p-1})<\lambda_{p_v+p_+}^+(Q_{p})$, by the induction hypothesis and \eqref{i01}, we have that
	\begin{equation}
	S_{p-1}(z)=(p-1)_v+(p-1)_+=p_v+p_+-1. \label{w}
	\end{equation}
	and hence
	\begin{equation}
	\sgn \ Q_{p-1}(z)=(-1)^{p_v+p_+-1}. \label{ww}
	\end{equation}
	By \eqref{qqq}, \eqref{w} and \eqref{ww},  one can obtain that
	$S_{p}(z)=p_v+p_+.$

	 If 
	$\lambda_{p_v+p_+}^+(Q_{p})<\lambda_{(p-1)_v+(p-1)_+}^+(Q_{p-1})$, we split the proof into three cases:\\
{\bf{Case 1.}} $z\in  (\lambda_{{p}_v+{p}_+}^+(Q_{p}),\lambda_{(p-1)_v+(p-1)_+}^+(Q_{p-1}))$. Due to our induction hypothesis and \eqref{i0}, one can obtain that
\begin{equation}
S_{p-1}(z)=(p-1)_v+(p-1)_+-1=p_v+p_+-1.  \label{si01}
\end{equation}
and hence
\begin{equation}
 \sgn \ Q_{p-1}(z)=(-1)^{p_v+p_+-1}.  \label{sgnqi0-1}
\end{equation}
By \eqref{qqq}, \eqref{si01} and  \eqref{sgnqi0-1} , we have that $S_{p}(z)=p_v+p_+$.    \\
{\bf{Case 2.}} $z\in  (\lambda_{(p-1)_v+(p-1)_+}^+(Q_{p-1}), +\infty).$ By our induction hypothesis and \eqref{i0}, we can obtain
\begin{equation}
S_{p-1}(z)=(p-1)_v+(p-1)_+=p_v+p_+  \label{q}
\end{equation}
 and hence
 \begin{equation}
 \sgn \ Q_{p-1}(z)=(-1)^{p_v+p_+}. \label{qq}
 \end{equation} 
  By \eqref{qqq}, \eqref{q} and \eqref{qq}, we have $S_{p}(z)=p_v+p_+$. \\
{\bf{Case 3.}} $z=\lambda_{(p-1)_v+(p-1)_+}^+(Q_{p-1})$.
Following the similar steps in the proof of {\bf{Case 3}} above, we obtain that  $S_{p}(z)=p_v+p_+$.

 Therefore, 	for $z\in (\lambda_{p_v+p_+}^+(Q_{p}),+\infty)$, if
$\lambda_{p_v+p_+}^+(Q_{p})<\lambda_{(p-1)_v+(p-1)_+}^+(Q_{p-1})$, we can conclude that
\begin{equation}
S_{p}(z)=p_v+p_+. \nonumber
\end{equation}

As a consequence, for $z\in (\lambda_{p_v+p_+}^+(Q_{p}),+\infty)$, one has
$S_{p}(z)=p_v+p_+$.
\end{proof}

\begin{lemma}\label{lemma5.9} For any $z\in \mathbb{R}$,
	if $\phi_+(z, x_j)\phi_+(z, x_{j+1})>0$, then $\phi_+(z, x)$ has no zero in the interval $[x_{j},x_{j+1}]$. If $\phi_+(z, x_j)\phi_+(z, x_{j+1})\le 0$, then $\phi_+(z, x)$ has exactly one zero in the interval $[x_{j},x_{j+i}]$.
\end{lemma}

\begin{proof}
We first show that if $\phi_+(z, x_j)\phi_+(z, x_{j+1})>0$, then $\phi_+(z, x)$ has no zero in the interval $[x_{j},x_{j+1}]$. Otherwise, let $y_1 \in [x_{j},x_{j+i}]$ be the zero of  $\phi_+(z, x)$.
Note that
for $x\in [x_{j},x_{j+1}]$,  there exists $A, B\in \mathbb{R}$, so that
\begin{equation}
\phi_{+}(z,x)=Ae^{\frac{x}{2}}+ Be^{-\frac{x}{2}}. \label{1}
\end{equation}
Then we have $Ae^{\frac{y_1}{2}}+ Be^{-\frac{y_1}{2}}=0$. That is, $B=-Ae^{y_1}$.
Substituting this into \eqref{1} and letting $x=x_j, x_{j+1}$, we have that
\begin{align}
\phi_+(z,x_j)=Ae^{\frac{x_j}{2}}-Ae^{y_1}e^{-\frac{x_{j}}{2}}, \
\phi_+(z,x_{j+1})=Ae^{\frac{x_{j+1}}{2}}-Ae^{y_1}e^{-\frac{x_{j+1}}{2}}. \nonumber 
\end{align}
Therefore, 
\begin{align}
\phi_+(z,x_j)\phi_+(z,x_{j+1})=A^2(e^{\frac{x_j}{2}}-e^{y_1-\frac{x_j}{2}})
(e^{\frac{x_{j+1}}{2}}-e^{y_1-\frac{x_{j+1}}{2}})\le 0. \nonumber
\end{align}
This is in contradiction with  the assumption that $\phi_+(z, x_j)\phi_+(z, x_{j+1})>0$.

 If $\phi_+(z, x_j)\phi_+(z, x_{j+1})\le 0$, then $\phi_+(z, x)$ has at least one  zero in the interval $[x_{j},x_{j+1}]$. We show $\phi_+(z, x)$ has exactly one  zero in the interval $[x_{j},x_{j+1}]$.
Otherwise, let $y_1, y_2\in [x_{j},x_{j+1}]$ be zeros of  $\phi_+(z, x)$.
Hence by \eqref{1}, one has
\begin{align}
Ae^{\frac{y_1}{2}}+ Be^{-\frac{y_1}{2}}=0, \
Ae^{\frac{y_2}{2}}+ Be^{-\frac{y_2}{2}}=0. \nonumber
\end{align}
Then one obtains that $A=B=0$. By \eqref{lianxu} and \eqref{tiaoyue}, we have $\phi_+(z,x)=0$ for all $x\in \mathbb{R}$. This contradicts \eqref{eq2.222}.
\end{proof}

\begin{theorem} \label{theorem5.9}
	$H$ has $n+n_v$ real and simple eigenvalues:
	\begin{equation}\nonumber
	\lambda_{n_v+n_-}^-<\cdots<
	\lambda_{1}^-<0<\lambda_{1}^+<\cdots<\lambda_{n_v+n_+}^+.
	\end{equation}
	The eigenfunction of the eigenvalue  $\lambda_{i}^{\pm}$ has exactly $i-1$ zeros.
\end{theorem}

\begin{proof}
	From Lemma \ref{lemma5.3} and Lemma \ref{lemma5.6},  we obtain that $H$ has $n+n_v$ real and simple eigenvalues.
	
If $n_v+n_+=(n-1)_v+(n-1)_+$.
 By \eqref{wz} and Lemma \ref{lemma5.6}, one has 
	$$0<\lambda_1<\lambda_{1}^+(Q_{n-1})<\cdots<\lambda_{n_v+n_+-1}^+(Q_{n-1})<\lambda_{n_v+n_+}^+<\lambda_{n_v+n_+}^+(Q_{n-1}).$$
	If $n_v+n_+=(n-1)_v+(n-1)_++1$. Then we have
		$$0<\lambda_1<\lambda_{1}^+(Q_{n-1})<\cdots<\lambda_{n_v+n_+-1}^+<\lambda_{n_v+n_+-1}^+(Q_{n-1})<\lambda_{n_v+n_+}^+<+\infty.$$
	Applying Lemma \ref{lemma5.7}, for both cases, we have that
	$S_{n-1}(\lambda_{i}^+)=i-1$. Then by  Lemma \ref{a}  and \eqref{definitionvarphii}, we know that $\phi_+(\lambda_i^+, x_1),\cdots, \phi_{+}(\lambda_i^+, x_n)$ change signs for $i-1$ times. 
By  Lemma \ref{lemma5.9}, $\phi_+(\lambda_i^+, x)$ has $i-1$ zeros in the interval $(x_1, x_n)$. Note that $\phi_+(\lambda_i^+,x)$ has no zero in the interval $(-\infty, x_1]$ and $[x_n, +\infty)$.
	 Therefore, the eigenfunction of  the eigenvalue $\lambda_i^+$ has exactly $i-1$ zeros.
	
	  By a similar argument,  the eigenfunction of the eigenvalue $\lambda_i^-$ also has exactly $i-1$ zeros.
\end{proof}
\begin{remark} \label{remark5.11}
Combining the fact that $S_{n-1}(\lambda_{1}^+)=0$ with \eqref{q0} and \eqref{qihephii},
 for any $x\in \mathbb{R}$, we have that
 $$\phi_{+}(\lambda_{1}^-,x), \phi_{+}(\lambda_{1}^+,x)>0.$$
\end{remark}
\begin{remark} \label{remark 6.11}
	Let us mention that if $n_v+n_-=0$, then  $n_v+n_+=n$. $H$ has $n$ positive and simple eigenvalues:
	\begin{equation}
	0<\lambda_{1}^+<\cdots<\lambda_{n}^+.  \nonumber
	\end{equation}
	The eigenfunction of the eigenvalues  $\lambda_{i}^{+}$ has exactly $i-1$ zeros.
\end{remark}

	By Lemma \ref{lemma3.11} and Theorem \ref{theorem5.9}, we have the following  corollary.
\begin{corollary}\label{corollary5.10}
{Assume that  $n+n_v\ge3$. Then for any $x \in \mathbb{R}$,  there exist $\lambda_i, \lambda_j\in \sigma(H)$, so that the corresponding eigenfunctions $\varphi_i$  and $\varphi_j$ satisfy
\begin{equation} \nonumber
\varphi_i(x)\ne 0,  \varphi_j(x)\ne 0.  
\end{equation}}
\end{corollary}

\appendix

\noindent \textbf{Acknowledgement.}  The authors are indebted to Guangsheng Wei for bringing this problem to their attention and for stimulating discussions.
 The authors would like to thank Dmitry Pelinovsky and Yiteng Hu for their valuable comments and additional references.

\end{document}